%% file: TopBS.tex
\documentclass[10pt]{amsart}
\usepackage{latexsym, amsmath,amssymb}
\usepackage{graphicx, color}

\usepackage{hyperref}

\usepackage{amsthm}

\newtheorem*{REMARK}{Remark}

\newcommand{\vertiii}[1]{{\left\vert\kern-0.25ex\left\vert\kern-0.25ex\left\vert #1 
    \right\vert\kern-0.25ex\right\vert\kern-0.25ex\right\vert}}

\newcommand{\tube}{{{\mathcal T}_{\lambda^{-\frac12}}(\gamma)}}
\newcommand{\Rtube}{{{\mathcal T}_R(\tilde \gamma)}}

\newcommand{\lapptq}{-{\Delta_{\tilde g}}}
\newcommand{\lat}{-{\Delta_{\tilde g}}}
\newcommand{\cosco}{\cos t \sqrt{ \lat}}
\newcommand{\lessim}{\lesssim}

\newcommand{\1}{{\rm 1\hspace*{-0.4ex}%
\rule{0.1ex}{1.52ex}\hspace*{0.2ex}}}

\renewcommand{\epsilon}{\varepsilon}
\newcommand{\newsection}[1]
{\subsection{#1}\setcounter{theorem}{0} \setcounter{equation}{0}
\par\noindent}

\newtheorem{theorem}{Theorem}

\newtheorem{lemma}[theorem]{Lemma}
\newtheorem{corr}[theorem]{Corollary}

\newtheorem{proposition}[theorem]{Proposition}
\newtheorem{deff}[theorem]{Definition}

\newcommand{\bth}{\begin{theorem}}
\newcommand{\ble}{\begin{lemma}}
\newcommand{\bcor}{\begin{corr}}

\newcommand{\bdeff}{\begin{deff}}

\newcommand{\bprop}{\begin{proposition}}
\newcommand{\ele}{\end{lemma}}
\newcommand{\ecor}{\end{corr}}
\newcommand{\edeff}{\end{deff}}

\newcommand{\eprop}{\end{proposition}}

\newcommand{\cd}{\, \cdot\, }

\newcommand{\Rn}{{\mathbb R}^n}

\newcommand{\la}{\lambda}

\newcommand{\e}{\varepsilon}

\renewcommand{\Pi}{\varPi}

\renewcommand{\epsilon}{\varepsilon}

\newcommand{\R}{{\mathbb R}}

\newcommand{\vp}{\varrho}


\begin{document}

\title[Toponogov's Theorem and improved estimates of eigenfunctions]{Concerning Toponogov's Theorem and   
logarithmic  improvement of
estimates of eigenfunctions}

\begin{abstract}
We use Toponogov's triangle comparison theorem from Riemannian geometry along with
quantitative scale oriented variants of classical propagation of singularities arguments
to obtain logarithmic improvements of the Kakeya-Nikodym norms introduced in 
\cite{SKN} for manifolds of nonpositive sectional curvature.  Using these and results from our paper \cite{BS15} we are able to obtain
log-improvements of $L^p(M)$ estimates for such manifolds when $2<p<\tfrac{2(n+1)}{n-1}$.  These in turn imply $(\log\la)^{\sigma_n}$, $\sigma_n\approx n$, improved lower bounds for $L^1$-norms
of eigenfunctions of the estimates of the second
author and Zelditch~\cite{SZ11}, and using a result from Hezari and the second author~\cite{HS},
under this curvature assumption, we are able to improve the lower bounds for the size
of nodal sets of Colding and Minicozzi~\cite{CM} by a factor of $(\log \la)^{\mu}$ for
any $\mu<\tfrac{2(n+1)^2}{n-1}$, if $n\ge3$.
\end{abstract}

\thanks{The authors were supported in part by the NSF grants DMS-1301717 and DMS-1361476, respectively.}

\author[M. D. Blair]{Matthew D. Blair}
\address{Department of Mathematics and Statistics, University of New Mexico, Albuquerque, NM 87131, USA}
\email{blair@math.unm.edu}
\author[C. D. Sogge]{Christopher D. Sogge}
\address{Department of Mathematics, Johns Hopkins University, Baltimore, MD 21093, USA}
\email{sogge@jhu.edu}

\maketitle

\newsection{Introduction}

The purpose of this paper is to show that if $(M,g)$ is a compact Riemannian manifold of dimension $n\ge 2$ then
one obtains logarithmically improved Kakeya-Nikodym bounds of eigenfunctions and appropriate eigenfunctions in all dimensions, as
well as log-improved restriction estimates in two-dimensions.  Using results from a companion paper \cite{BS15}, we deduce that we have
improved $L^p(M)$ estimates for the range $2<p<\tfrac{2(n+1)}{n-1}$.  These imply log-improvements of the  $L^1(M)$  bounds
of the second author and Zelditch~\cite{SZ11}, which in turn imply log-improvements (assuming nonpositive curvature) of the
lower bounds of Colding and Minicozzi~\cite{CM} for the size of nodal sets.  The assumption of nonpositive curvature is needed
since, except for the last estimate regarding nodal sets, all of the estimates are saturated by the highest weight spherical harmonics,
on $S^{n-1}$,
$k^{\frac{n-1}4} \text{Re } (x_1+ix_2)^k$, $k=1,2,\dots$.  They also are saturated by the Gaussian beams constructed by Ralston~\cite{Ral},
which cannot exist on manifolds of nonpositive curvature.

We shall consider $L^2$-normalized eigenfunctions on a Riemannian manifold $(M,g)$ of dimension
$n\ge 2$.  So we are considering functions satisfying
$$(\Delta_g+\la^2)e_\la=0, \quad \text{and } \, \, \int_M|e_\la|^2 \, dV_g=1,$$
where $\Delta_g$ and $dV_g$ of course are the Laplace-Beltrami operator and volume element associated
with the metric $g$ on $M$, respectively.  So in our case $\la$ is the frequency of the eigenvalue, and also
$e_\la$ is an eigenfunction with eigenvalue $\la$ of $P=\sqrt{-\Delta_g}$.

In what follows, we let $\varPi$ denote the space of all geodesic segments of length 
one, assuming that $\text{Inj }M\ge 10$, where
$\text{Inj }M$ denotes the injectivity radius of $(M,g)$.  
Otherwise, we shall modify things so that the geodesics in $\varPi$ are of length $\text{Inj }M/10$.  This
slight ambiguity is caused by the fact that we shall want to assume eventually that the sectional curvatures
of $(M,g)$ are pinched below by $-1$ to simplify the statement of the cone comparison result that we shall employ, which play a
critical role in our analysis.

If
$\gamma\in \varPi$,  ${\mathcal T}_\varepsilon(\gamma)$ denotes a geodesic tube of width $\varepsilon$ about $\gamma\in \varPi$,
then the Kakeya-Nikodym
norm of our eigenfunctions  defined by
\begin{equation}\label{1.1} \vertiii{e_\la}_{{KN}} =\Bigl(\, \sup_{\gamma \in \varPi} \int_{\tube}|e_\la|^2 \, dV_g \, \Bigr)^{\frac12},
\end{equation}
were introduced by one of us in \cite{SKN}, following earlier related work of Bourgain~\cite{Bo}, as a way of controlling the $L^p(M)$ norms
of eigenfunctions.  Although not explicitly stated, the inequalities proved in \cite{SKN} yield
\begin{equation}\label{1.2}
\|e_\la\|_{L^4(M)} \le C\la^{\frac18}\vertiii{e_\la}_{KN}^{\frac14}, \quad \text{if } \, n=2,
\end{equation}
and hence improvements over the trivial estimate
$$\vertiii{e_\la}_{KN}\le 1,$$
would yield improvements over the second author's earlier bounds \cite{Seig} $\|e_\la\|_4=O(\la^{\frac18})$ (saturated on $S^2$).
By interpolating with the $L^6$ estimate there, one would also get improvements for the full range $2<p<6$.  The trivial Kakeya-Nikodym  bounds
were improved by the second author and Zelditch in \cite{SZKN}, who showed that, in two dimensions, one has $\vertiii{e_\la}_{KN}=o(1)$,
and, as a result, the improvement $\|e_\la\|_4=o(\la^{\frac18})$, under the assumption of nonpositive curvature.  These sorts
of results were extended to higher dimensions by the authors in \cite{BS}.  In all cases, though, even though we could show that
the various norms were relatively small as $\la\to \infty$, there was no control on the rate of decay of the Kakeya-Nikodym
norms or on the way that the constants in the $L^p$ improvements over the ones in \cite{Seig} go to zero as $\la\to \infty$.

The purpose of this paper is to establish that there are improvements in turns of powers of $\log \la$ for all of these things.  Our main
result, which along with estimates in a companion paper, yields these improved bounds. 

\begin{theorem}\label{mainthm}  Suppose $(M,g)$ has nonpositive sectional curvatures.  Then
\begin{equation}\label{K1}
\sup_{\gamma\in \varPi}\int_\tube |e_\la|^2 \, dV \lesssim c(\lambda),
\end{equation}
for $\la\gg 1$ with
$$c(\lambda)=
\begin{cases}
(\log \la)^{-\frac12}, \quad \text{if  } \, n=2
\\
(\log \la)^{-1} \log\log \la, \quad \text{if  }\, n=3
\\
(\log \la)^{-1}, \quad \text{if  }\, n\ge 4.
\end{cases}
$$
Moreover, if $n=2$, we have
\begin{equation}\label{K2}
\sup_{\gamma\in \varPi} \int_\gamma |e_\la|^2 \, ds \le C\la^{\frac12} c(\lambda).
\end{equation}
The above estimates hold as well when $e_\lambda$ is replaced by a quasi-mode satisfying
\begin{equation}\label{qm} \|\psi_\la\|_{L^2(M)} + (\log\la/\la) \bigl\|(\Delta_g+\lambda^2)\psi_\la\bigr\|_{L^2(M)}\le 1.\end{equation}
\end{theorem}

The estimate \eqref{K2} is stronger than \eqref{K1} when $n=2$.  Moreover,
it represents a logarithmic improvement over the two-dimensional restriction estimates
of Burq, G\'erard and Tzvetkov~\cite{BGT}.  As we pointed out in \cite{BS} in higher dimensions,
$n\ge 4$, restriction estimates, as opposed to the Kakeya-Nikodym tube estimates as in \eqref{K1}, are too
singular to control $L^p$ norms.  In these dimensions they are saturated by eigenfunctions matching the
profile of zonal spherical harmonics, rather than highest weight spherical harmonics which
saturate the Kakeya-Nikodym norms in \eqref{K1}.  In this case, it is also a bit more straightforward to obtain
logarithmic improvements over the geodesic restriction estimates in \cite{BGT}.  This was done
by Chen~\cite{Ch} when $n>3$.  Chen and one of us \cite{CS} also showed that when $n=3$ one could improve
on the universal bounds (saturated in this case by both zonal harmonics and highest weight spherical
harmonics) assuming that $(M,g)$ is of constant nonpositive curvature.  Whether one can get logarithmic
improvements for either constant or variable nonpositive curvature is an interesting open problem in this dimension.  

The proof of these results will follow closely the general scheme introduced in \cite{SKN} and \cite{SZKN}.  The new 
ingredient is that we  are using classical triangle comparison theorems from Riemannian geometry
to make tighter pseudo-differential cutoffs allowing us to use the time-averaging method over
logarithmic  time intervals (as opposed to large ones basically not depending on $\la$).  We
use the universal cover of $(M,g)$ to break up the operators that we use to obtain
our estimates into a number of pieces.  We 
use quantitative and scale oriented variants of classical propagation of singularities arguments
(i.e., integration by parts)
 to
handle what turn out to be the small, but numerous, ``error'' terms which 
arise from terms associated with a portion of the universal cover not contained 
 in a natural cone of small apeture about the
geodesic.
We can handle  the relatively few remaining ``local terms'' in standard
ways.

Before we turn to the proof of Theorem let us state a couple of corollaries of our
main theorem and results from our companion paper \cite{BS15}.  The first
concerns logarithmically improved $L^p$ norms of eigenfunctions and appropriate quasi-modes.

\begin{corr}\label{Lpcorr}  Assume, as above, that $(M,g)$ is a compact $n\ge2$ dimensional
manifold with nonpositive sectional curvatures.  Then for any $2<p<\tfrac{2(n+1)}{n-1}$
there is a number $\sigma(p,n)>0$ so that
\begin{equation}\label{lp}
\|e_\la\|_{L^p(M)}\lesssim \la^{\frac{n-1}2(\frac12-\frac1p)} \, \bigl(\log \la\bigr)^{-\sigma(p,n)}.
\end{equation}
Furthermore, if $\tfrac{2(n+2)}n <p<\tfrac{2(n+1)}{n-1}$, one can take
\begin{equation}\label{sigma}   \quad
\sigma(p,n)=
\begin{cases}
\tfrac{n+1}{n-1}(\tfrac1p-\tfrac{n-1}{2(n+1)}), \, \, \, \text{if } \, \,  n\ge 4,
\\ \\ 
\text{ any } \, \, \sigma(p,3)<2(\tfrac1p-\tfrac14), \, \, \, \text{if } \, \, n=3,
\\  \\
\tfrac32(\tfrac1p-\tfrac16), \, \, \, \text{if } \, \, n=2. \end{cases}
\end{equation}
The above estimates hold as well when $e_\la$ is replaced by a quasi-mode
satisfying \eqref{qm}.
\end{corr}

Let us show how we can obtain these $L^p$ norms for eigenfunctions and postpone
the discussion of quasi-modes for a bit.  To prove \eqref{lp} we shall use \eqref{K1}
and the following estimate from our companion paper \cite{BS15},
\begin{equation}\label{mlKN}
\|e_\la\|_{L^p(M)}\lesssim \la^{\frac{n-1}2(\frac12-\frac1p)}\vertiii{e_\la}_{KN}^{\frac{2(n+1)}{n-1}(\frac1p-\frac{n-1}{2(n+1)})},
\quad \text{if } \, \, \tfrac{2(n+2)}n<p<\tfrac{2(n+1)}{n-1}.
\end{equation}
Using this estimate and our Kakeya-Nikodym bounds \eqref{K1} we immediately
get \eqref{lp} for the range $\tfrac{2(n+2)}n<p<\tfrac{2(n+1)}{n-1}$, and the estimates
from the remaining range $2<p\le \tfrac{2(n+2)}n$ follow from interpolation
with the trivial bound $\|e_\la\|_2\le 1$.

Using these estimates we can in turn get improvements for lower bounds of the 
$L^1$ norms of eigenfunctions and the size of their nodal sets under our curvature
assumptions:

\begin{corr}
Assume that $(M,g)$ is as above.  Then
\begin{equation}\label{l1}
\la^{-\frac{n-1}4} (\log\la)^{\mu} \lesssim \|e_\la\|_{L^1(M)},
\end{equation}
for any $\mu<\mu_n$ with
\begin{equation*}   \quad
\mu_n=
\begin{cases}
\tfrac{(n+1)^2}{n-1}, \quad \text{if } \, \, n\ge 3
\\ \\
\tfrac{(n+1)^2}{2(n-1)}, \quad \text{if } \, \, n=2.
\end{cases}
\end{equation*}
Consequently, if $e_\la$ is a real-valued eigenfunction and $|Z_\la|$ denotes
the $(n-1)$-dimensional Hausdorff measure of its nodal set, $Z_\la = \{x: \, e_\la(x)=0\}$, we have
\begin{equation}\label{nodal}
\la^{1-\frac{n-1}2} \bigl(\log \la\bigr)^{2\mu} \lesssim |Z_\la|,
\end{equation}
when $\mu<\mu_n$.  In particular, when $n=3$, $(\log \la)^{r}\lesssim |Z_\la|$ for
all $r<16$.
\end{corr}

The lower bound of $\la^{1-\frac{n-1}2}\lesssim |Z_\la|$ is due to Colding and Minicozzi~\cite{CM} and it is
the best known lower bound for general $C^\infty$ manifolds.  An alternate proof of this lower bound was later found by 
the author and Zelditch~\cite{SZnod2}.  In the real analytic case Donnelly and Fefferman~\cite{DF} showed that
$|Z_\la|\approx \la$.  Earlier Yau~\cite{Yau} had conjectured this bound for general smooth Riemannian manifolds.

To prove \eqref{l1}, we first notice that, by
H\"older's inequality
$$1=\|e_\la\|_{L^2(M)}\le \|e_\la\|_{L^1(M)}^{\frac{p-2}{2(p-1)}}\|e_\la\|_{L^p(M)}^{\frac{p}{2(p-1)}},
$$
and so
$$\la^{-\frac{n-1}4}
\, \bigl(\la^{-\frac{n-1}2(\frac12-\frac1p)}\|e_\la\|_{L^p(M)}\bigr)^{-\frac{p}{p-2}}\le \|e_\la\|_{L^1(M)}.$$
To get \eqref{l1} we plug the estimates \eqref{lp} into this inequality and realize that the resulting
lower bounds improve as $p\searrow \tfrac{2(n+2)}n$.  The power $\mu_n$ represents
what we would obtain if \eqref{lp} were valid at the endpoint, which thus gives us \eqref{l1}.
To prove \eqref{nodal} we just use the lower bound
of Hezari and the second author~\cite{HS},
$$\la \|e_\la\|_{L^1(M)}^2 \lesssim |Z_\la|,$$
which was proved using ideas from an earlier work of the second author and
Zelditch~\cite{SZ11}.

In a recent work, Hezari and Rivi\`ere~\cite{HR} were able to obtain log-improvements for
$L^p$-norms of a subsequence of eigenfunctions of density one assuming that the sectional
curvatures of $(M,g)$ are strictly negative.  Our results relax the latter condition to our
 assumption that the curvatures be nonpositive, but, more significantly, we are able to
handle all eigenfunctions and not just ones corresponding to a density one subsequence of
eigenvalues.  The results of Hezari and Rivi\`ere~\cite{HR} are based on obtaining non-trivial 
log-improvements of $L^2$-norms of eigenfunctions over shrinking balls (which had also been obtained earlier and independently
by Han~\cite{H}) and then using an estimate that relates such estimates to $L^p$-norms
(see also \cite{Sball} for the latter).  Over the years, there have been many works on improving $L^p$-norms for relatively large exponents
$p>\tfrac{2(n+1)}{n-1}$, including \cite{HT}, \cite{STZ}, \cite{SZ}, \cite{SZRA} and \cite{SZRA2}.
 Hassell and Tacy~\cite{HT} extended B\'erard's sup-norm estimate, implicit in \cite{Berard}, by
 showing that there are $(\log\la)^{-1/2}$ improvements over the bounds of \cite{Seig} under
 the assumption of nonpositive curvature for all $p>\tfrac{2(n+1)}{n-1}$.  The endpoint case where $p=\tfrac{2(n+1)}{n-1}$
 is the remaining case where general results concerning improvements 
 for the full sequence
 of eigenfunctions under natural geometric assumptions
 remains open.

\bigskip


\newsection{Some reductions and tools}

To prove \eqref{K1}, we shall use the fact that if $\rho\in {\mathcal S}(\R)$ satisfies
\begin{equation}\label{K3}
\rho(0)=1 \quad \text{supp }\widehat \rho \subset [-1/2,1/2],
\end{equation}
then
\begin{equation}\label{K4}\rho(T(\la-P))e_\la = e_\la \quad \text{if } \, \, P=\sqrt{-\Delta_g}.
\end{equation}
Consequently, we would have the estimate \eqref{K1} for eigenfunctions if we could show that we have
\begin{equation}\label{K1'} 
\|\rho(T(\la-P))f\|_{L^2(\tube)}\lesssim \sqrt{c(\la)} \, \|f\|_{L^2(M)},
\quad T\approx \log \lambda,
\end{equation}
where the constants involved are independent of $\gamma\in \varPi$.

Let us also see why \eqref{K1'} implies that \eqref{K1} also holds for quasi-modes
satisfying \eqref{qm}.  
Let $\{e_j\}$ denote an orthonormal basis of eigenfunctions with eigenvalues $\la_j\to \infty$ and let
$E_j$ denote the projection onto the $j$th eigenspace.
Then since $\rho(0)=1$ and since $T\approx \log \la$, it is not difficult to see that if we define the spectral
projectors associated to windows of width $(\log \la)^{-1}$, i.e.,
$$E_{[\la,\la+(\log\la)^{-1}]} \,  f=\sum_{\la_j\in [\la,\la+(\log\la)^{-1}]}E_jf,$$
we have
\begin{equation}\label{qmbd}
\|E_{[\la,\la+(\log \la)^{-1}]} \, f\|_{L^2(\tube)} \lesssim \sqrt{c(\la)} \|f\|_{L^2(M)}.
\end{equation}
Using this it is very simple to deduce that \eqref{K1} also must hold for functions $\psi_\la$ satisfying \eqref{qm}.
(See \cite{SZqm} for similar arguments.)  

The $L^p$ estimates for quasi-modes are proved by a similar argument.
Theorem 1.1 in  \cite{BS15} shows that 
for $T$ as above
\begin{multline*}\|\rho(T(\la-P))f\|_{L^p(M)}
\lesssim \la^{\frac{n-1}2(\frac12-\frac1p)}\vertiii{\rho(T(\la-P))f}_{KN}^{\frac{2(n+1)}{n-1}(\frac1p-\frac{n-1}{2(n+1)})},
\\ \text{if } \, \, \tfrac{2(n+2)}n<p<\tfrac{2(n+1)}{n-1}.
\end{multline*}
 If we combine this with \eqref{K1'}, we deduce that for exponents such exponents we have
 $$\|\rho(T(\la-P))f\|_{L^p(M)} \lessim 
  \la^{\frac{n-1}2(\frac12-\frac1p)} \, \bigl(\log \la\bigr)^{-\sigma(p,n)} \|f\|_{L^2(M)}.$$
  Since $T\approx \log\la$ and $\rho(0)=1$, this in turn implies that 
  $$\|E_{[\la,\la+(\log \la)^{-1}]} \, f\|_{L^p(M)} \lesssim  \la^{\frac{n-1}2(\frac12-\frac1p)} \, \bigl(\log \la\bigr)^{-\sigma(p,n)} \|f\|_{L^2(M)}.$$
By interpolating with the trivial $L^2$ estimate we see that we also get bounds of this type for $2<p\le \tfrac{2(n+2)}n$.
These $L^2\to L^p$ bounds for 
$(\log\la)^{-1}$ sized spectral projector operators are easily seen to imply that \eqref{lp} is also valid for quasi-modes satisfying \eqref{qm}
(see \cite{SZqm}).


To prove \eqref{K1'}, we note that this is equivalent to showing that
\begin{multline}\label{K1''} 
\|\chi(T(\la-P)) f\|_{L^2(\tube)}  \lesssim c(\la) \, \|f\|_{L^2(M)}, \\
\text{if supp }f\subset \tube \, \, \text{and }
\quad T\approx \log \lambda,
\end{multline}
with $\chi=|\rho|^2$.  We note then that
\begin{equation}\label{K3'}\tag{2.1$'$}
\widehat \chi\subset [-1,1].
\end{equation}

We shall take $T$ to be $c\log \la$ where $c>0$ will be a small constant
depending on a lower bound for the sectional curvatures of $(M,g)$,
among other things.
There is
no loss of generality in assuming, as we shall, that they satisfy
\begin{equation}\label{K5}
K \ge -1,
\end{equation}
and we also recall that we are assuming that they are everywhere nonpositive.

We shall then use  quantitative microlocal analysis  bounds and the following geometric fact,
which is a consequence of Toponogov's triangle comparison theorem, to prove
Theorem~\ref{mainthm}.

\begin{proposition}[{\bf Toponogov}]\label{topprop}  
Equip $\Rn$ with a metric $\tilde g$ of nonpositive curvature
and assume that the sectional curvatures, $K$, of $(\Rn,\tilde g)$ also satisfy
$$K\ge -1.$$
Let $\tilde \gamma(t)$, $t\in \R$, be a geodesic with $\tilde \gamma(0)=P$.  Given $T\gg 1$,
let $C(\theta;T)$, $\theta\ll 1$, denote the set of points $Q\in B_{\tilde g}(P;T)$ which lie on a geodesic
though $P$ which intersects $\tilde \gamma$ of angle $\le \theta$.
Thus, $C(\theta;T)$ is the intersection of the geodesic ball $B_{\tilde g}(P;T)$ with the  cone of aperture $\theta$ about $\tilde \gamma$ with
vertex $P$.
 Fix $R>0$.  Then if 
$${\mathcal T}_R(\tilde \gamma)=\{x\in \Rn: \, d_{\tilde g}(x,\tilde \gamma)\le R\},$$
we have that 
\begin{equation}\label{K6}
C(\theta_T;T)\subset {\mathcal T}_R(\tilde \gamma), \quad \text{if } \, \, \sin \tfrac12\theta_T = \frac{\sinh \tfrac12R}{\sinh T}, 
\quad \text{if } \, T>0.
\end{equation}
\end{proposition}


\medskip

\noindent{\bf Remark:}  To simplify the notation we are assuming throughout that the nonpositive curvature is pinched below by $-1$.  If we assumed
the nonpositive sectional curvatures were bounded below by $-\kappa^2$, then the proof of Proposition~\ref{topprop} 
which we shall present
also gives that 
$$C(\theta_{T,\kappa};T)
\subset \mathcal{T}_R(\tilde \gamma), \quad \text{if } \, \sin\tfrac12\theta_{T,\kappa}=\frac{\sinh \tfrac\kappa2 R}{\sinh \kappa T}, \quad \text{if } \, \, T>0.$$
Sending $\kappa \searrow 0$, despite being slightly weaker, one essentially recovers the familiar ``sine equals opposite over hypothesis'' in Euclidean geometry.

\medskip

The other main tool that we shall use, as in our earlier related works \cite{SZKN}, \cite{BS15}, is a scale oriented microlocalization about
the unit cosphere bundle $S^*\gamma\subset S^*M$ associated with our $\gamma\in \varPi$.  We may work in local coordinates so that $\gamma$ is just
$$\{(t,0,\dots,0): \, 0\le t\le 1\}.$$
Fix then $\alpha\in C^\infty_0(\R)$ satisfying $\alpha(s)=1$ for $|s|\le 1$ and $\alpha(s)=0$ for $|s|\ge 2$ and define compound symbols
\begin{multline}\label{PDO}
Q_{\theta,\lambda}(x,y,\xi)=\alpha\bigl(\theta^{-1}d_g(x,\gamma)\bigr) \, \alpha\bigl(\theta^{-1}d_g(y,\gamma)\bigr) \, \alpha\bigl(\theta^{-1}|\xi'|/|\xi|\bigr) \, \Upsilon(|\xi|/\la), 
\\ \xi'=(\xi_2,\dots,\xi_n),
\end{multline}
where $\Upsilon\in C^\infty(\R)$ is assumed to satisfy
\begin{equation}
\label{dyadic}
\Upsilon(s)=1, \, \, s\in [c_0,c_0^{-1}], \quad \Upsilon(s)=0, \, \, s\notin[\tfrac{c_0}2, 2c_0^{-1}],
\end{equation}
with $c_0>0$ a small but fixed number to be specified later.  We then define the associated integral operators $Q_{\theta,\la}$ with kernels
$$Q_{\theta,\la}(x,y)=(2\pi)^{-n}\int_{\Rn} e^{i\langle x-y,\xi\rangle} \, Q_{\theta,\la}(x,y,\xi)\, d\xi$$
expressed in our local coordinates about $\gamma$.

In what follows we shall take $\la^{-\frac12+\delta_0}\le \theta \ll 1$ for some $\delta_0>0$ to also be specified later.
Note that
\begin{multline}\label{K10}
|D^{\alpha_1}_{x,y}D^{\alpha_2}_\xi Q_{\theta,\la}(x,y,\xi)|\le C_{\alpha_1,\alpha_2}\theta^{-|\alpha_1|-|\alpha_2|} \la^{-|\alpha_2|}, \, \, \forall \, \alpha_1,\alpha_2,
\\
\text{and } \, \, |\partial^j_{\xi_1}Q_{\theta,\la}(x,y,\xi)|\le C_j\la^{-j}.
\end{multline}
From this and a simple integration by parts argument we deduce that we have the uniform bounds for such $\theta$ and $\la \gg 1$
$$|Q_{\theta,\la}(x,y)|\le C_N \theta^{n-1}\la^n\bigl(1+\la|x_1-y_1| +\theta\la |x'-y'|)^{-N}, \quad \forall \, N=1,2,\dots.$$
Consequently, we have the uniform bounds
\begin{multline}\label{kernel}
\sup_x \int |Q_{\theta,\la}(x,y)|\, dy, \quad \sup_y \int |Q_{\theta,\la}(x,y)| \, dx \le C,
\\
\text{and } \, \, |Q_{\theta,\la}(x,y)|\le C_N\theta^{n-1}\la^n\bigl(1+\theta\la d_g(x,y)\bigr)^{-N}, \quad N=1,2,3,\dots.
\end{multline}

We then shall use the following local result which is valid for {\em all compact Riemannian manifolds},
which also does not require the support assumptions in \eqref{K1''}.
It is
based on the escape times of balls of radius $(\theta\la)^{-1}$ exiting $\tube$ or $\gamma$ if
they are at  traveling at unit speed along geodesics of angle $\theta$ from $\gamma$.

\begin{proposition}[{\bf Escape times}]\label{tatprop}  Fix a compact Riemannian manifold $(M,g)$ of dimension
$n\ge2$, and let
 $a\in C^\infty_0((-1,1))$.  Then if the constant $c_0>0$ in \eqref{dyadic} is small enough
\begin{equation}\label{K11}
\Bigl\|\int a(t) e^{-it\la}  \bigl(I-Q_{\theta,\la}\bigr)  e^{itP}f\, dt\Bigr\|_{L^2(\tube)}
\le C_{\delta_0}\la^{-\frac14}\theta^{-\frac12}\|f\|_{L^2(M)},
\end{equation}
with $C_{\delta_0}$ independent of $\theta\ge\la^{-\frac12+\delta_0}$, if $0<\delta_0<1/2$ is fixed
and $\la$ is large.  Moreover, when $n=2$ and  $c_0>0$ is sufficiently small, we have, for any $\e>0$,
\begin{equation}\label{K11'}
\Bigl\|\int a(t) e^{-it\la}  \bigl(I-Q_{\theta,\la}\bigr)  e^{itP}f\, dt\Bigr\|_{L^2(\gamma)}
\le C_{\delta_0,\e}\theta^{-\frac12-\e}\|f\|_{L^2(M)},
\end{equation}
with $L^2(\gamma)$ denoting the norm taken with respect to arc length measure over
our $\gamma\in \varPi$.  For a given $n$, the constants in \eqref{K11} and \eqref{K11'} also only depend on the size of
finitely many derivatives of $a$.
\end{proposition}

The two-dimensional estimate \eqref{K11'} is related to 
estimates of Greenleaf and Seeger~\cite{GS} for Fourier integral operators associated to one-sided folding canonical
relations and
trace estimates for the wave equation of Tataru~\cite{Tat}.

\newsection{The main argument}

We shall postpone the proof of the two propositions until after this section.  Now let us show how they imply our Kakeya-Nikodym estimates.

\medskip

Let us split up our operators $\chi_\la$ into two pieces.  The first,
\begin{equation}\label{K12}
\chi^\theta_\la = Q_{\theta,\la} \circ \chi(T(\la-P)) =\frac1{2\pi T}\int_{-T}^T
\widehat \chi(t/T) e^{-i\la t} \, Q_{\theta,\la} \circ e^{itP} \, dt,
\end{equation}
denotes the microlocalization near the geodesic, which should be thought of as the
``main'' term, while the ``remainder'', $R^\theta_\la =\chi_\la -\chi_\la^\theta$, is given
by
\begin{equation}\label{K13}
R^\theta_\lambda = (I-Q_{\theta,\la})\circ \chi(T(\la-P))=
\frac1{2\pi T} \int_{-T}^T \widehat \chi(t/T) e^{-i\la t} \, \bigl(I-Q_{\theta,\la}\bigr)\circ
e^{itP} \, dt.
\end{equation}

Using \eqref{K11} it is very easy to handle the remainder 
on {\em any} manifold.  No curvature assumptions are needed.
Choose $a\in C^\infty_0(\R)$ so that $\sum_{-\infty}^\infty a(t-k)\equiv 1$.  Then since
$e^{ikP}$ maps $L^2(M)$ to itself with norm 1, \eqref{K11} implies that we have the 
uniform bounds
$$\frac1{2\pi T} \Bigl\|\int a(t-k)\widehat \chi(t/T) e^{-i\la t}(I-Q_{\theta,\la})e^{itP}f \, dt
\Bigr\|_{L^2(M)}\le CT^{-1}\la^{-\frac14}\theta^{-\frac12}\|f\|_{L^2(M)}.$$
Since, in view of \eqref{K3'} the left side is zero if $|k|\ge 2T$, if we sum over these bounds
we deduce that
\begin{equation}\label{K14}
\|R^\theta_\lambda f\|_{L^2(\tube)}\le C\la^{-\frac14}\theta^{-\frac12}\|f\|_{L^2(M)}.\end{equation}

We shall always assume that
\begin{equation}\label{K15}
\theta \ge \la^{-\delta}, \quad \text{for some } \, 0<\delta \ll 1/2,
\end{equation}
which we can achieve by taking the $\delta_0$ in the Proposition~\ref{tatprop} to be close
to $1/2$.
Consequently, if we assume that $\delta$ in \eqref{K15} is small enough
we get
\begin{equation}\label{K16}
\|R^\theta_\lambda f\|_{L^2(\tube)} \le C\la^{-\frac18}\|f\|_{L^2(M)},
\end{equation}
which is much better than the bounds posited in \eqref{K1''}.
We note here that, in addition to making no curvature assumptions, we are not requiring
the support assumptions in \eqref{K1''}.  

On account of \eqref{K16}, we would have \eqref{K1''} if we could show that
\begin{equation}\label{K17}
\|\chi^\theta_\la \|_{L^2(\tube) \to L^2(\tube)}\lesssim c(\la), \quad \text{if } \, T=c\log\lambda,
\end{equation}
where $c>0$ will be a small constant, chosen, for instance, so that the $\delta>0$ in \eqref{K15} is small.
Also, the first part of \eqref{K17} means that we assume that the operators satisfy the bounds
when, as in \eqref{K1''}, we assume that the functions are supported in $\tube$.
Since $\chi(T(\la+P))$ has a smooth kernel with $O(\la^{-N})$ bounds on all derivatives independent
of $T\ge 1$, to prove \eqref{K17}, by Euler's formula, it suffices to show that
\begin{equation}\label{K17'} \tag{3.6$'$}
\| \tilde\chi^\theta_\la \|_{L^2(\tube) \to L^2(\tube)}\lesssim c(\la), \quad \text{if } \, T=c\log\lambda,
\end{equation}
where
\begin{equation}\label{K12'}\tag{3.1$'$}
\widetilde \chi^\theta_\la = \frac1{\pi T}\int
\widehat \chi(t/T) e^{-i\la t} \, Q_{\theta,\la} \circ \cos t\sqrt{-\Delta_g} \, dt,
\end{equation}

We have switched from $\exp(it\sqrt{-\Delta_g})$ to $\cos t\sqrt{-\Delta_g}$ so that we can
use the Hadamard parametrix  and the Cartan-Hadamard theorem to lift the calculations
that will be needed for \eqref{K17'} up to the universal cover $(\Rn,\tilde g)$ of $(M,g)$. 
This is the approach that was used in \cite{BS} and \cite{SZKN}.

 Let $\{\alpha\}=\Gamma$
denotes the group of deck transformations preserving the associated covering map $\kappa: \Rn \to M$ coming
from the exponential map from $\gamma(0)$ associated with the metric $g$ on $M$.  The metric $\tilde g$
then is its pullback via $\kappa$.  Choose also a Dirichlet fundamental domain, $D\simeq M$, for $M$ centered
at the lift $\tilde \gamma(0)$ of $\gamma(0)$.  We shall let $\tilde \gamma(t)$, $t\in \R$, denote the lift of the
geodesic $\gamma(t)$, $t\in \R$, containing the unit segment $\gamma(t)$, $0\le t\le 1$ around which our tube
is centered.  We shall work in  geodesic normal coordinates vanishing at $\tilde \gamma(0)$ and we may assume that
$$\tilde \gamma(t)=\{(t,0,\dots, 0): \, t\in \R\}.$$

Let $\Gamma_{{\mathcal T}_R(\tilde \gamma)}\subset\Gamma$ be all those deck transformations for which
$$\alpha(D)\cap {{\mathcal T}_R(\tilde \gamma)} \ne \emptyset, \, \, \, \text{where } \, \,
\, \, R=100 \cdot \text{diam }D.$$
Here ${\mathcal T}_R(\tilde \gamma)$ denotes the $R$-tube about the geodesic in $(\Rn,\tilde g)$, and, 
since $D$ is a Dirichlet domain
$$R\approx \text{Inj }M.$$
We measure distances in $(\Rn,\tilde g)$ using its Riemannian distance function $d_{\tilde g}(\cd, \cd)$.  The distance
function on $(M,g)$ is similarly denoted by $d_{g}(\cd, \cd)$.

Following \cite{SZKN}, we recall also that if $\tilde x$ denotes the lift of $x\in M$ to $D$, then we have the following formula
$$\bigl(\cos t\sqrt{-\Delta_g}\bigr)(x,y)=\sum_{\alpha\in \Gamma}
\bigl(\cos t\sqrt{-\Delta_{\tilde g}}\bigr)(\tilde x,\alpha(\tilde y)).$$
Consequently,
\begin{equation}\label{K12''}\tag{3.1$''$}
\widetilde \chi^\theta_\la =\sum_{\alpha\in \Gamma} U_\alpha^{\theta,\la},
\end{equation}
where $U_\alpha^{\theta,\la}$ is the operator with kernel $U^{\theta,\la}(\tilde x,\alpha(\tilde y))$, where
\begin{multline}\label{K18}
U^{\theta,\la}(\tilde x,\tilde y)=\frac1{\pi T}\int \widehat \chi(t/T) \, e^{i\la t} \, \bigl(Q_{\theta,\la} \circ \cos(t\sqrt{-\Delta_{\tilde g}})\bigr)(\tilde x, \tilde y) \, dt\\
=\left(Q_{\theta,\la}\circ \Bigl[  \, \frac1{\pi T}\int \widehat \chi(t/T) \, e^{i\la t} \, \bigl( \cos(t\sqrt{-\Delta_{\tilde g}})\bigr) \, dt \, \Bigr] \right)(\tilde x, \tilde y).
\end{multline}

If we let $K(\tilde x,\tilde y)$ denote the kernel of the operator in the square brackets, i.e.,
$$K(\tilde x,\tilde y)=\frac1{\pi T}\int \widehat \chi(t/T) \, e^{i\la t} \, \bigl( \cos(t\sqrt{-\Delta_{\tilde g}})\bigr)(\tilde x,\tilde y) \, dt,$$
then, as we shall show, one can use the Hadamard parametrix to obtain the uniform bounds
\begin{equation}\label{K19}
|K(\tilde x,\tilde y)|\le CT^{-1}\la^{\frac{n-1}2} \bigl(d_{\tilde g}(\tilde x,\tilde y)\bigr)^{-\frac{n-1}2}, \quad \text{if } \, \, d_{\tilde g}(\tilde x,\tilde y)\ge1,
\end{equation}
provided that, as we are assuming, $T=c\log\la$ with $c>0$ sufficiently small.
The estimates also hold when $d_{\tilde g}(\tilde x,\tilde y)\le 1$, but we shall need this for now.

We claim that these size estimates along with Young's inequality and \eqref{kernel} imply that we have the uniform bounds
\begin{equation}\label{K19'}\tag{3.8$'$}
\|U_\alpha^{\theta,\la}\|_{L^2(\tube)\to L^2(\tube)} \le CT^{-1} \bigl(1+d_{\tilde g}(0,\alpha(0))\bigr)^{-\frac{n-1}2},
\end{equation}
when $\alpha\ne Identity$, and a slightly different argument will be needed to show that the bounds
also hold for $\alpha=Identity$.  We shall postpone the simple proof of \eqref{K19} until after we introduce the
Hadamard parametrix.

Recall that, by construction, the kernel $Q_{\theta,\la}(x,y)$ vanishes if the distance from either $x$ or $y$ to
$\{\gamma(t): \, 0\le t\le 1\}$ is  larger than $2\theta$.
Therefore, \eqref{K15} and  \eqref{kernel}, for $x,y\in \tube$, we have
\begin{align*}
|U^{\theta,\la}(\tilde x,\alpha(\tilde y))|&\le C\sup_{x \in {\mathcal T}_{2\theta}(\gamma), \, y\in \tube}|K(\tilde x,\alpha(\tilde y))|
\\
&\le CT^{-1}\la^{\frac{n-1}2}\bigl(d_{\tilde g}(0,\alpha(0))\bigr)^{-\frac{n-1}2},
\end{align*}
due to our
assumption \eqref{K15}.  This and the $L^1$ estimates for the kernel of $Q_{\theta,\la}$ account for the first
inequality here, and the second follows from the fact that $d_{\tilde g}(\tilde x,\alpha(\tilde y))\approx d_{\tilde g}(0,\alpha(0))$, 
if $x,y\in {\mathcal T}_{2\theta}(\gamma)$, 
due
to our assumptions about $\text{Inj }M$.  By using the above estimate and Young's inequality, one obtains \eqref{K19}.

To show it is valid for $\alpha=Identity$, choose $\eta\in C^\infty_0(\R)$ satisfying $\eta(s)=1$, $|s|\le 2$ and $\eta(s)=0$, $|s|\ge 3$.  Then
 we can write
\begin{equation*}
K(\tilde x,\tilde y)=\frac1{\pi T}\int \eta(t) \widehat \chi(t/T) e^{i\la t}\bigl(\cos t\sqrt{-\Delta_g}\bigr)(\tilde x,\tilde y) \, dt
+R(\tilde x,\tilde y),\end{equation*}
where
\begin{equation*}
R(\tilde x, \tilde y)=\frac1{\pi T} \int \bigl(1-\eta(t)\bigr) \widehat\chi(t/T) e^{i\la t} \bigl(\cos t\sqrt{-\Delta_{\tilde g}}\bigr)( \tilde x, \tilde y) \, dt,
\end{equation*}
since $(\cos t\sqrt{-\Delta_{\tilde g}})(\tilde x,\tilde y)=(\cos t\sqrt{-\Delta_{g}})(x,y)$ if $d_g(x,y)\le 3$, due to Huygens principle and our assumptions about
$\text{Inj }M$.  The operator
with kernel
equal to the first term  in the formula for $K(\tilde x,\tilde y)$ is obviously bounded from $L^2(M)$ to itself with norm $O(T^{-1})$ since $\cos t\sqrt{-\Delta_g}$ has
norm one.  Later we shall show, also using the Hadamard parametrix, that
\begin{equation}\label{K19''}\tag{3.8$''$}
|R(\tilde x,\tilde y)|\le CT^{-1}, \quad \text{if } \, \, d_g( x, y) \le 2,
\end{equation}
 provided, as we are assuming, $T=c\log\la$ with $c>0$ sufficiently small.  Using these facts
  and the fact that $\|Q_{\theta,\la}\|_{L^2(M) \to L^2(M)}\le C$ (by \eqref{kernel}), one
  deduces that \eqref{K19'} also must hold for $\alpha=Identity$, which completes its proof
  (apart from showing that \eqref{K19} and \eqref{K19'} are valid as we shall do later).

Since there are only $O(1)$, ``translates,'' $\alpha(D)$, of $D$ that intersect any geodesic ball 
with arbitrary center
of radius $R$ (which was fixed earlier),\footnote{One sees this assertion 
by noting that, since the $\alpha$ are isometric, the volume, $v$, of $D$ agrees with that
of any $\alpha(D)$, $\alpha\in \Gamma$.  Similarly, if $d$ denotes the diameter of our Dirichlet domain $D$, then $\alpha(D)$ has the same diameter.  Let $B_{\tilde g}(P;R)$
denote the geodesic ball of radius $R$ about some point $P$.  Then if $\alpha(D)\cap B_{\tilde g}(P;R) \ne \emptyset$,
it follows from the triangle inequality that $\alpha(D)\subset B_{\tilde g}(P;R+d)$.  As the $\alpha(D)$ are disjoint,
the number of such domains intersecting the ball is therefore bounded from above by $\text{Vol}_{\tilde g}\bigl(B_{\tilde g}(P;R+d)\bigr)/v$.
Since volume comparison theorems (see, e.g. \cite{Chavel}) and our curvature assumptions imply that $\text{Vol}_{\tilde g}\bigl(B_{\tilde g}(P;R+d)\bigr)$ is bounded
by the the volume of balls with  this radius in hyperbolic space, ${\mathbb H}^n$, the assertion follows.}
 it follows that
$$\#\{\alpha \in \Gamma_{{\mathcal T}_R(\tilde \gamma)}: \, d_{\tilde g}(0,\alpha(0))\in [2^k,2^{k+1}]\} \le C 2^k.
$$
This is because one can cover the set $\{x \in \Gamma_{{\mathcal T}_R(\tilde \gamma)}: \, d_{\tilde{g}}(0,x) \in [2^k,2^{k+1}]\}$ with $O(2^k)$ balls of radius $R$.
Thus, we deduce from \eqref{K3'} and \eqref{K19'} that
\begin{equation*}
\sum_{\alpha\in \Gamma_{{\mathcal T}_R(\tilde \gamma)}} \|U_\alpha^{\theta,\la}\|_{L^2(\tube)\to L^2(\tube)}
\lesssim 
T^{-1} 
\sum_{1 \leq 2^k \leq T} 2^k 2^{-k\frac{n-1}2 }
\lesssim c(\la),
\quad
\text{if } \, \, T=c\log\la,
\end{equation*}
where the constants $c(\la)$ are as in \eqref{K1}.  Recall that $\bigl(\cos t\sqrt{-\Delta_{\tilde g}}\bigr)(x,y)=0$ if
$d_{\tilde g}(x,y)>t$. Here, as in  the rest of this section, we are dropping the tildes from our various points $x$, $y$, etc., since all the
calculations will be done in $(\Rn,\tilde g)$.

Since there are $O(\exp(c_0T))$ nonzero $U_\alpha^{\theta,\la}$ for some fixed $c_0$, we deduce that we would finish
matters and obtain \eqref{K17'} if we could choose the constant $c$ in definition of $T$ so that
for large enough $\la$ we have.
\begin{equation}\label{K20}
\|U_\alpha^{\theta,\la}\|_{L^2(\tube)\to L^2(\tube)} \le \la^{-1} \quad \text{if } \, \, \alpha \notin \Gamma_{{\mathcal T}_R(\tilde \gamma)}.
\end{equation}

We shall need to finally use Toponogov's theorem to do this.   We note that, by Proposition~\ref{topprop}, we have
$$\bigl\{r\tfrac\xi{|\xi|}: \, \, \bigl|\tfrac\xi{|\xi|} - (1,0,\dots,0)\bigl| \le c_R\theta_T, \, \, |r|\le T \bigr\}\subset
{{\mathcal T}_R(\tilde \gamma)}, \, \, \, \text{if } \,  \, \theta_T=e^{-T}.$$
Recall that the  $x$-gradient
of the Riemannian distance function, $d_{\tilde g}(x,y)$ points in the direction of the tangent vector at $x$  of
the geodesic connecting $x$ and $y$.  So the last assertion just means that 
$$\min_{\pm}\Bigl| \frac{\nabla_x d_{\tilde g}(x,y)}{|\nabla_x d_{\tilde g}(x,y)|} \pm \1\Bigr|\ge c_R\theta_T,
\quad \text{if } \, \, x=0\, \, \text{and } \, \, y\notin {{\mathcal T}_R(\tilde \gamma)},$$
with $\1$ denoting the vector $(1,0,\dots,0)$.  Repeating ourselves, this is just because the geodesic
cones of aperture $\approx \theta_T$ (measured by the metric $\tilde g$)
with vertex $0$ and central directions $\pm \1=\pm\tfrac{d}{dt}\tilde \gamma(t)$, $t=0$, are contained in the intersection of the ball of radius $T$ centered at
our origin and ${{\mathcal T}_R(\tilde \gamma)}$.  The same will remain true at any $x$ point on our
unit length geodesic $\{\tilde \gamma(t):0\le t\le 1\}$ since we would just be replacing $T$ by a radius
in $[T-1,T+1]$.  Hence, if we choose $c_\delta=c>0$ in the definition of $T=c\log \lambda$ to be
small enough so that $c_R\theta_T=\lambda^{-\delta/2}$, we have the crucial fact that
\begin{equation}\label{K21}\min_{\pm}\Bigl| \frac{\nabla_x d_{\tilde g}(x,y)}{|\nabla_x d_{\tilde g}(x,y)|} \pm \1\Bigr|\ge \la^{-\delta/2},
\quad \text{if } \, \, x=\gamma(t), \, \, 0\le t\le 1, \, \, \text{and } \, \, y\notin {{\mathcal T}_R(\tilde \gamma)}.
\end{equation}
We can take $\delta$ to be any fixed small positive number by adjusting this constant $c$.  In particular,
we shall want it to be small enough so that both \eqref{K15} and \eqref{K16} are both valid.  We should also point out
that every time we reduce the size of $c$ this has the effect of 
improving the Toponogov lower bound \eqref{K6}
and so the
previous steps of the proof will not be invalidated when we make future reductions.


 The pseudo-differential
cutoff $Q_{\theta,\la}$ occurring in the definition \eqref{K18}  of $U_\alpha^{\theta,\la}$ involves microlocalizing at an angle 
 $\theta=\theta_\la$ of size $\la^{-\delta}$ about our geodesic, which is much smaller than the one occurring in \eqref{K21}
 if $\la$ is large.  This will
allow us to show that for $\la \gg 1$ we have
\begin{equation}\tag{3.9$'$}\label{K20'}
|U^{\theta,\la}( x, y)|\le \la^{-1} \quad \text{if } \, \, 
d_{\tilde g}\bigl( x, \, \{\tilde \gamma(t): \, 0\le t\le 1\}\bigr)\le \la^{-\frac12}, \, \, \text{and } \, \, 
 y\notin {{\mathcal T}_R(\tilde \gamma)},
\end{equation}
which is more than adequate for obtaining \eqref{K20}.  As noted before, $U_\alpha^{\theta,\la}( x, y)$ vanishes identically
when $d_{\tilde g}( x, y)>T$.

Note that, 
because of the coordinates
we are using,
by \eqref{PDO}, the compound
symbol of our pseudo-differential cutoff $Q_{\theta,\la}$ 
satisfies
\begin{multline}\label{K22}
Q_{\theta,\la}(x,y,\xi)=0\quad \text{if } \, \, 
d_{\tilde g}(x,\{(t,0,\dots,0): 0\le t\le 1\})\ge C\la^{-\delta}, 
\\
d_{\tilde g}(y,\{(t,0,\dots,0): 0\le t\le 1\})\ge C\la^{-\delta}, 
\, \, \, \text{or } \, \, \min_{\pm}\bigl| \pm\tfrac\xi{|\xi|}-\1\bigr|\ge C\lambda^{-\delta},
\end{multline}
for some constant $C$ where $\1=(1,0,\dots,0)$
denotes the direction of our unit speed geodesic $\tilde \gamma(t)=(t,0,\dots,0)$.  Similarly, by \eqref{K10}, since we are taking $\theta=\lambda^{-\delta}$, the symbol also satisfies
satisfy the size estimates
\begin{equation} 
\label{size}
|D^{\alpha_1}_{x,y}D^{\alpha_2}_\xi Q_{\theta,\la}(x,y,\xi)|\le C_{\alpha_1,\alpha_2}
\lambda^{\delta(|\alpha_1|+|\alpha_2|)} (1+|\xi|)^{-|\alpha_2|}.
\end{equation}

At this point we shall follow the argument in B\'erard~\cite{Berard} and use the fact that
since $(\Rn,\tilde g)$ has nonpositive curvature as well as curvature pinched below by $-1$ we can
use the Hadamard parametrix to write the kernel of $\cos t
\sqrt{\lapptq}$ as a Fourier integral whose symbol and phase have derivatives growing at most exponentially
in terms of the geodesic distance from the origin.  There also will be a remainder with this property, but
this will trivial to deal with.

The phase functions will just involve the Riemannian distance functions
$d_{\tilde g}(x,y)=r$.  In addition to using \eqref{K21}, we shall need to use the fact
that when derivatives are taken with respect to our coordinate system we 
have for every multi-index $\alpha$
\begin{equation}\label{K23}
D^\alpha_x d_{\tilde g}(x,y)\le C_\alpha \exp(c_\alpha r),
\end{equation}
for some constants $C_\alpha$ and $c_\alpha$. This follows from Proposition 3 and Lemma 4 on p. 274 in B\'erard~\cite{Berard}.
Note that we get from this and \eqref{K21} that, after possibly reducing the size of the constant $c$
in the definition of $T$, we may assume in addition to \eqref{K21} that for large enough $\lambda$ we have
\begin{multline}\label{K21'}\tag{3.10$'$}
\min_{\pm}\Bigl| \frac{\nabla_w d_{\tilde g}(w,y)}{|\nabla_w d_{\tilde g}(w,y)|} \pm \1\Bigr|\ge \la^{-\delta/2},
\\ \text{if } \, \, 
d_{\tilde g}\bigl(w,\{\tilde \gamma(t), \, 0\le t\le 1\}\bigr) =O(\la^{-\frac12}),  \, \, \text{and } \, \, y\notin {{\mathcal T}_R(\tilde \gamma)},
\end{multline}
since we are assuming that $\delta$ is smaller than $1/2$.

Next, following B\'erard\footnote{We are using the phase functions $\theta(d_{\tilde g} \pm t)$ for our Fourier integrals for the sake of convenience,
instead of $\theta(d_{\tilde g}^2-t^2)$ as in \cite{Berard}, but our formulation follows from B\'erard's and an obvious change of variables.}~\cite{Berard}, we
can use the Hadamard parametrix to write for, say, $|t|$ or $d_{\tilde g}(x,y)$  larger than $1$,
\begin{equation}\label{K24}
\bigl(\cos t\sqrt{\lat}\bigr)(x,y)=\sum_{\pm}\int_{\R} e^{i\tau(d_{\tilde g}(x,y)\pm t)} \, a_\pm(x,y,\tau) \, d\tau
+R(t,x,y),
\end{equation}
where, if $|t|<T$ as well and if $N_0\in {\mathbb N}$ is fixed
\begin{multline}\label{K25}
|\bigl(\tfrac\partial{\partial \tau}\bigr)^ja_\pm(x,y,\tau)|\le C_j\bigl(1+d_{\tilde g}(x,y)\bigr)^{-\frac{n-1}2} \, 
(1+|\tau|)^{\frac{n-1}2-j}, \quad \text{and } 
\\
\bigl| D^{\alpha_1}_x\bigl(\tfrac\partial{\partial \tau}\bigr)^j a_\pm(x,y,\tau)\bigr| \le C_{j,\alpha_1}\exp(c_{N_0}T) \,
(1+|\tau|)^{\frac{n-1}2-j}, \, \, \, 0<|\alpha_1|\le N_0,
\end{multline}
and we may assume that the remainder satisfies 
\begin{equation}\label{K26}
\bigl| \bigl(\tfrac\partial{\partial t}\bigr)^j R(t,x,y)\bigr| \le C_{N_0}\exp(c_{N_0}T), \quad
\text{if } \, \,  j\le N_0.
\end{equation}
The various constants here are independent of $T$, but the exponential rate of growth in this parameter, may depend
on the number of space-time derivatives taken.  These bounds were established in \cite{Berard}.

Fix now $\vp\in C^\infty(\R)$ satisfying $\vp=1$ on $[-R/2,R/2]$ and $\text{supp }\vp\subset [-2R,2R]$.  Then, by
Huygens' principle and \eqref{K18}
\begin{multline}\label{K18'}\tag{3.7$'$}
U^{\theta,\la}(x,y)=\frac1{\pi T}\int (1-\vp(t)) \widehat \chi(t/T) \, e^{i\la t} \,  \, \bigl(Q_{\theta,\la} \circ \cos(t\sqrt{-\Delta_{\tilde g}})\bigr)(x, y) \, dt, 
\\ 
\text{if }\, x\in \tube, \, \, \text{and } \, \,  y\notin {\mathcal T}_R(\tilde \gamma).
\end{multline}


If we let $R_t$ denote the integral operator with kernel $R(t,x,y)$, then, by 
\eqref{kernel} and \eqref{K26}, we have
\begin{equation}\label{K26'}\tag{3.16$'$}
\bigl|\partial_t^j \bigl(Q_{\theta,\la} \circ R_t\bigr)(x,y)\bigr|\le C\exp(cT), \quad j=0,1,2.
\end{equation}
By
a simple integration by parts argument in $t$,
we see from this that if we replace $\cosco$ by $R_t$ in \eqref{K18'}, we would obtain a kernel $R(x,y)$ 
satisfying
\begin{equation}\label{K27}
|R(x,y)|\le B\exp(BT) \, \la^{-2},
\end{equation}
for some constant $B$. 

Since this can be made small compared to $\la^{-1}$ for large $\la$ if the constant
$c$ in the definition of $T$ is chosen to be small enough, we would obtain \eqref{K20'} if we could obtain
similar bounds when we replace $\cosco$ by each of the Fourier integral operators whose kernels
are the two main terms in \eqref{K24} coming from the sum over $\pm$.

To simplify the calculation that will be involved let us make a couple of trivial reductions.

The first involves the pseudo-differential cutoff.  We recall that the kernel of $Q_{\theta,\la}$ is given
by 
$$Q_{\theta,\la}(x,w) =\int e^{i\langle x-w,\xi\rangle} Q_{\theta,\la}(x,w,\xi)\, d\xi,$$
where the symbol satisfies \eqref{K22} and \eqref{size}.  Since we are assuming that $\delta$ is smaller than
$1/2$ it follows that if 
$$\widetilde Q_{\theta,\la}(x,w)=\vp\bigl(\la^{\frac12}d_{\tilde g}(x,w)\bigr)
\int e^{i\langle x-w,\xi\rangle} Q_{\theta,\la}(x,w,\xi)\, d\xi,$$
then $R_\theta(x,w)=Q_{\theta,\la}(x,w)-\widetilde Q_{\theta,\la}(x,w)$ is also supported near our unit geodesic and
satisfies
$$|\partial^\alpha_{x,w}R_\theta(x,w)|\le C_{N,\alpha} \la^{-N} \quad \forall \, \alpha, N,$$
by \eqref{kernel}.
Therefore, by an argument similar to the one just given, if we replace $Q_{\theta,\la}$ by $\widetilde Q_{\theta,\la}$
and $\cosco$ in \eqref{K18'} by the sum of the two Fourier integrals in \eqref{K24}, 
then the difference between the kernel in  \eqref{K18'} and the resulting kernel will satisfy the bounds in \eqref{K27}.  Thus, in what follows,
we may replace $Q_{\theta,\la}$ by $\widetilde Q_{\theta,\la}$ and ignore the remainder term  in  \eqref{K24} in our calculations.  We have made this reduction to
make it simpler to apply \eqref{K21'}.

For the Fourier integral operators we shall note that we can make a Littlewood-Paley decomposition.  Specifically,
if $\beta\in C^\infty_0(\R)$, satisfies
\begin{equation}\label{littlewoodpaley}
\beta(s)=1, \, \, \text{if } \, s\in [1/2,2], \quad \text{and } \, \, \beta(s)=0, \, \, \text{if } \, \, s\notin [1/4,4],
\end{equation}
then we note that
$$T^{-1}\Bigl| \int \bigl(1-\vp(t)\bigr) \, \widehat \chi(t/T) \, \bigl(1-\beta(|\tau|/\la)\bigr) e^{i t(\la\pm \tau)} \, dt\, \Bigr|
\le C_N(\la+|\tau|)^{-N}, \quad \forall \, N,$$
where the $C_N$ are independent of $T\ge 1$.  As a result, we may also multiply the symbols in the 
Fourier integrals by $\beta(|\tau|/\la)$ in our calculation at the expense of introducing another error
satisfying the bounds which are better than those in \eqref{K27}.

Summarizing, in proving the remaining estimate 
to establish \eqref{K20'},
we may replace $Q_{\theta,\la}$ by $\widetilde Q_{\theta,\la}$ and
the two Fourier integrals in \eqref{K24} by the ones where the symbol is multiplied by $\beta(|\tau|/\la)$.
If we also abuse notation a bit and multiply the symbol of $ Q_{\theta,\la}$ by a smooth factor, but
still write it as $ Q_{\theta,\la}$,
to take into account that the integrations are given with respect to the volume element, we have reduced
matters to showing that 
\begin{multline}\label{K28}
T^{-1}\left| \iiint \bigl(1-\vp(t)\bigr) \, \widehat \chi(t/T) \, e^{it(\la\pm\tau)}e^{ix\cdot \xi}
 \, b_\pm(\la,x,w,y,\tau,\xi) \, e^{i\phi(w,y,\tau,\xi)} \, dw d\tau d\xi dt\right|
\\
\le B\exp(BT)\la^{-2} , 
\quad \text{if } \, \, 
d_{\tilde g}\bigl( x, \, \{\tilde \gamma(t): \, 0\le t\le 1\}\bigr)\le \la^{-\frac12} \, \, \text{and } \, \, 
 y\notin {{\mathcal T}_R(\tilde \gamma)},
 \end{multline}
 where the symbol here is given by
 $$b_\pm(\la,x,w,y,\tau,\xi)
 =\vp(\la^{\frac12}d_{\tilde g}(x,w)) \, 
\widetilde Q_{\theta,\la}(x,w,\xi) \, a_\pm(w,y,\tau) \,  \beta(|\tau|/\la),$$
 and the phase here is given by
 $$\phi(w,y,\tau,\xi)=\tau d_{\tilde g}(w,y)-w\cdot \xi.$$
 To proceed, we note that, due to the $\vp(\la^{\frac12}\cd)$ cutoff, the first condition in \eqref{K21'} is fulfilled
 on the support of the integrand
 and we are assuming the other one in \eqref{K27}.  Therefore, if we recall the $\xi$-support assumptions
 in \eqref{K22} for $Q_{\theta,\la}$, we can use \eqref{K21'} to see that, on the support of the integrand, there must be a positive constant $c_0$ so that
 for large $\la$
 \begin{multline*}
 |\nabla_w \phi(w,y,\tau,\xi)| \ge c_0\la^{-\frac\delta2} 
 \bigl(|\la|+|\tau|+|\xi| \bigr), \quad \text{if } \, \, b_\pm \ne 0, 
 \\
  d_{\tilde g}\bigl( x, \, \{\tilde \gamma(t): \, 0\le t\le 1\}\bigr)\le \la^{-\frac12}, \, \, \text{and } \, \, 
 y\notin {{\mathcal T}_R(\tilde \gamma)}.
 \end{multline*}
 Since $Q_{\theta,\la}$, $d_{\tilde g}$ and $a_\pm$   satisfy the size estimates in \eqref{size}, \eqref{K22} and \eqref{K25},
 respectively, we can integrate by parts a finite number of times to obtain \eqref{K28}.  We choose $N_0$
 in \eqref{K25} to be large enough to handle the number of integration by parts which depends only on the dimension.
 
 \medskip
 
 To wrap up matters, apart from proving the two Propositions, 
 to prove the Kakeya-Nikodym estimates
 we still have to prove \eqref{K19} and \eqref{K19''}.  We shall use
 a variation on the argument that we just gave which is considerably easier.
 
 To prove the former we use \eqref{K24}.  If we replace $(\cosco)(x,y)$ in the definition of $K(x,y)$ by the remainder
 term in \eqref{K24}, the resulting expression will clearly be $O(\exp(c_0T))$, which clearly is smaller than
 the right side of \eqref{K19} if $T=c\log \la$ with $c>0$ small provided that $d_{\tilde g}(x,y)\le T$, as we may
 assume since $K(x,y)$ vanishes otherwise by Huygens' principle.  If we replace $(\cosco)(x,y)$ by the main term
 in the Hadamard parametrix, \eqref{K24}, the resulting expression equals half of the sum over $\pm$ of
 $$\int_{-\infty}^\infty
 e^{i\tau d_{\tilde g}(x,y)}
  \chi\bigl(T(\la\pm \tau)\bigr) \, a_\pm(x,y,\tau) \, d\tau,$$
 and, using the first part of \eqref{K25} with $j=0$, this is also seen to satisfy the bounds in \eqref{K19} since
 $\chi\in {\mathcal S}(\R)$.
 
 To prove \eqref{K19''}, one uses a similar argument.  Using \eqref{K26'} for $j=0,1$ one sees by an integration by parts argument in $t$
 that the contribution of the remainder term to \eqref{K19''} is $O(\la^{-1}\exp(cT))$, which is much better than the bounds
 in \eqref{K19''} if $T=c\log\la$ with $c$ small.  Since we are assuming that $|d_{\tilde g}\pm t|\ge 1$ on the support of the
 integral defining $R(x,y)$, if we plug the main term of the Hadamard parametrix \eqref{K24} and integrate by parts in 
 $\tau$, we see that its contribution is $O(\la^{-N})$ for any $N$ since the Fourier transform of
 $t\to T^{-1}\bigl(1-\eta(t)\bigr)\widehat \chi(t/T)$ is bounded and rapidly decreasing at infinity.  Thus, its contribution also is much
 better than the bounds posited in \eqref{K19''}, which finishes its proof.
 
 \medskip
 
 Let us conclude this section by showing how these arguments imply the two-dimensional restriction estimates \eqref{K2}.
 If $\rho\in {\mathcal S}(\R)$ is as in \eqref{K3}, the bounds for eigenfunctions would follow from showing that
 we have the uniform bounds
 \begin{equation}\label{K2'} 
 \|\rho(T(\la-P))f\|_{L^2(\gamma)}\le C(\la/\log\la)^{\frac14} \|f\|_{L^2(M)}, \, \, \, \text{if } \, \gamma\in \varPi,
 \, \, \text{and } \, \, T\approx \log \la.
 \end{equation}
 Since $\rho(0)=1$ this implies
 $$\|E_{[\la,\la+(\log\la)^{-1}]}f\|_{L^2(\gamma)}\lesssim (\la/\log\la)^{\frac14}\|f\|_{L^2(M)},
 $$
 which in turn, by a routine argument (cf. \cite{SZqm}) implies that quasi-modes satisfying the condition \eqref{qm} also
 satisfy the bounds in \eqref{K2}.
 
 If one uses the argument from  the  first part of the proof of \eqref{K1} along with \eqref{K11'},  one finds that
 $$\|(I-Q_{\theta,\la})\circ \rho(T(\la-P))f\|_{L^2(\gamma)}\lesssim_\e \theta^{-\frac12-\e}\|f\|_{L^2(M)},
 $$
 which is better than the bounds in \eqref{K2'} assuming, as we shall, that $\theta=\la^{-\frac12+\delta_0}$ for some
 $\delta_0>0$ and $\e<\delta_0$.  Thus, we have reduced the proof of \eqref{K2'} to showing that for suitable $\theta$ we have
 $$\|Q_{\theta,\la}\circ \rho(T(\la-P))f\|_{L^2(\gamma)}\le C(\la/\log\la)^{\frac14}\|f\|_{L^2(M)},$$
 assuming that, as before, $T=c\log\la$, for some fixed $c>0$.
 
 By a routine $TT^*$ argument if $\chi=|\rho|^2$, this is equivalent to showing that
 \begin{multline}\label{R1}
 \Bigl(\int_0^1\Bigr|\int_0^1 \bigl(Q_{\theta,\la}\circ \chi(T(\la-P)) \circ Q^*_{\theta,\la}\bigr)(\gamma(s),\gamma(s')) \, h(s')\, ds'
 \Bigr|^2 ds\Bigr)^{\frac12}
 \\
 \le C(\la/\log\la)^{\frac12}\|h\|_{L^2([0,1])}.
 \end{multline}
 
 Since $\chi(T(\la+P))(x,y)=O(\la^{-N})$, arguing as before, we see that, modulo a trivial error, we can express the kernel here
 as $\sum_{\alpha\in \Gamma}U^{\theta,\la}_\alpha$ where now
 $$U^{\theta,\la}_\alpha(s,s')=\Bigl(Q_{\theta,\la}\circ \chi\bigl(T(\la-P)\bigr)(\cd, \alpha(\cd))\circ Q^*_{\theta,\la}\Bigr)(\gamma(s),\gamma(s')).$$
 If $\alpha\ne Identity$,  the earlier arguments yield the uniform bounds
 $$|U^{\theta,\la}_\alpha(s,s')|\le CT^{-1}\la^{\frac12}\bigl(d_{\tilde g}(0,\alpha(0))\bigr)^{-\frac12}, \quad 0\le s,s'\le 1,$$
 and so 
 \begin{multline}\label{R2}
 \Bigl(\int_0^1\Bigl| \int_0^1 U_\alpha^{\theta,\la}(s,s')\, h(s') \, ds'\Bigr|^2 ds\Bigr)^{\frac12}
 \\
 \le CT^{-1}\la^{\frac12}\bigl(1+d_{\tilde g}(0,\alpha(0)\bigr)^{-\frac12}
 \|h\|_{L^2([0,1])}, \quad \alpha \ne Identity.
 \end{multline}
 
 We claim that the same bounds hold for $\alpha=Identity$.  Indeed if, as before, $\eta\in C^\infty_0(\R)$ satisfies $\eta(s)=1$, $|s|\le 2$
 and $\eta(s)=0$, $|s|\le 3$, it follows that when $\alpha$ is the identity the kernel can be split as
 \begin{multline}\label{R3}
 U^{\theta,\la}_{Id}(s,s')=\frac1{\pi T}
\Bigl(Q_{\theta,\la}\circ \Bigl(\, \int \eta(t)\widehat \chi(t/T) \cos t\sqrt{-\Delta_g} \, dt\, \Bigr) \circ Q^*_{\theta,\la}\Bigr)(\gamma(s),\gamma(s'))
\\
+R(s,s'),
\end{multline}
where, by the proof of \eqref{K19''}
$$|R(s,s')|\le CT^{-1}, \quad \text{if } \, \, 0\le s,s'\le 1.$$
Since $\text{Inj }M\ge 10$ one can use a parametrix for $\cos t\sqrt{-\Delta_g}$ and stationary phase and argue like in the proof of 
\cite[Lemma 5.1.3]{SFIO} to see that the first term in the right side of \eqref{R3} is $O(\la^{{\frac12}}(d_g(\gamma(s),\gamma(s')))^{-\frac12})
=O(\la^{\frac12}|s-s'|^{-\frac12})$, and so, by Young's inequality, \eqref{R2} is valid for $\alpha=Identity$ as well.

If $U^{\theta,\la}_\alpha$ are the operators with these kernels our earlier arguments give, analogous to what happened before,
$$\sum_{\alpha\in \Gamma_{{\Rtube}}}\|U^{\theta,\la}_\alpha\|_{L^2(\gamma)\to L^2(\gamma)}\lesssim 
T^{-1}\la^{\frac12}
\sum_{1 \leq 2^k \leq T} 2^k 2^{-\frac{k}2 }
\lesssim (\la/\log\la)^{\frac12}, \quad \text{if } \, \, T\approx \log\la.$$

By the support properties of $Q_{\theta,\la}(x,y)$, \eqref{kernel} and \eqref{K20'}, all the other terms are $O(\la^{-1})$ if the various
parameters are chosen as before.  Hence, if, as before, $T=c\log\la$ with $c>0$ sufficiently small, we have
$$\sum_{\alpha\notin \Gamma_{{\Rtube}}}\|U^{\theta,\la}_\alpha\|_{L^2(\gamma)\to L^2(\gamma)} \le 1,$$
which along with the previous bound (and the fact that $\chi(T(\la+P))(\gamma(s),\gamma(s'))=O(1)$) gives us
\eqref{R1}, which completes the proof of the restriction estimates.
 
 \medskip
 
 This completes the proof of Theorem~\ref{mainthm}, apart from the two Propositions, which we shall
 handle next.  
 
\newsection{Proof of  Proposition~\ref{topprop}} \hspace*{\fill} \\
Recall that we are trying to show that
$$C(\theta_T;T)\subset {\mathcal T}_R(\tilde \gamma), \quad \text{if } \, \sin \tfrac12\theta_T =\frac{\sinh \tfrac12 R}{\sinh T}.$$
Where we can take $P$, the vertex of $C(\theta_T;T)$, to be $\tilde \gamma(0)$, and $\tilde \gamma(t)$, $t\in \R$, to be the 
geodesic we have just been working with.  $C(\theta; T)$ then is the intersection of the geodesic ball of radius $T>0$ about our origin
with the cone of aperture $\theta$ about $\tilde \gamma$, which has coordinates $\{(t,0,\dots,0): \, t\in \R\}$
in the geodesic normal coordinate system we are using.
Also,
${\mathcal T}_R(\tilde \gamma)$ denotes the closed tube of fixed radius $R>0$ about $\tilde \gamma$. 

Since $T\to \theta_T$ is monotonically decreasing, it suffices to show that a point $Q$ with coordinates $T\omega$, $\omega\in S^{n-1}$,
belongs to ${\mathcal T}_R(\tilde \gamma)$ if the angle, $\sphericalangle (\omega,\1)$, is $\le \theta_T$.  In other words, to obtain
\eqref{K6}, it suffices to show that
\begin{equation}\label{K6'} 
\Sigma(T;\theta_T)\subset {\mathcal T}_R(\tilde \gamma),
\end{equation}
if $\Sigma(T,\theta)$ denotes all points $Q$ with coordinates $T\omega$ satisfying $\sphericalangle (\omega,\1)\le \theta$.

Clearly $\Sigma(T;\theta)\subset {\mathcal T}_R(\tilde \gamma)$ when $\theta$ is very small (depending on $T$).  So choose the maximal
$\Theta_T\le \pi/2$ so that $\Sigma(T;\theta)\subset {\mathcal T}_R(\tilde \gamma)$ when $0<\theta<\Theta_T$.  It follows
that there must be a point $Q$ with coordinates,   $T\omega_0$, satisfying $\sphericalangle(\omega_0,\1)=\Theta_T$ and $d_{\tilde g}(Q,\tilde \gamma)=R$.
Also, \eqref{K6'} is valid when $\theta_T$ is replaced by $\Theta_T$ and so we would have \eqref{K10'} and be done if we could show that
\begin{equation}\label{K6''} 
\Theta_T \ge \theta_T.
\end{equation}

At this point, we shall use Toponogov's theorem.  First consider the geodesic triangle,
$\triangle^{\tilde g}_{T,\Theta_T}$, in $(\Rn,\tilde g)$ with vertices $Q$
and the point with coordinates $0$ and the $P$ point with coordinates $(T,0,\dots,0)$.  It is an isosceles triangle since
the geodesics connecting the point with coordinates $0$ with $P$ and $Q$, respectively, each have length $T$.  The point
$P$ lies on $\tilde \gamma$ and hence if $\tilde \gamma_{opp}$ is the third side of our geodesic triangle,
which connects $P$ and $Q$, we must have that its length, $\ell(\tilde \gamma_{opp})$ satisfies
 $$\ell(\tilde \gamma_{opp})=d_{\tilde g}(P,Q)\ge R,$$
 since, as we pointed out before, we must have  
$d_{\tilde g}(Q,\tilde \gamma)=R$.  The angle at the vertex whose coordinates are the origin, by construction, is $\Theta_T$,
and the two sides passing through it each have length $T$.
The third side of our isosceles triangle, $\tilde \gamma_{opp}$, is called a ``Rauch hinge.''  See Figure~\ref{Rauch}.

\begin{figure}[h]
\hspace*{-2.5cm}                                                           
\resizebox{4.2in}{2.1in}{
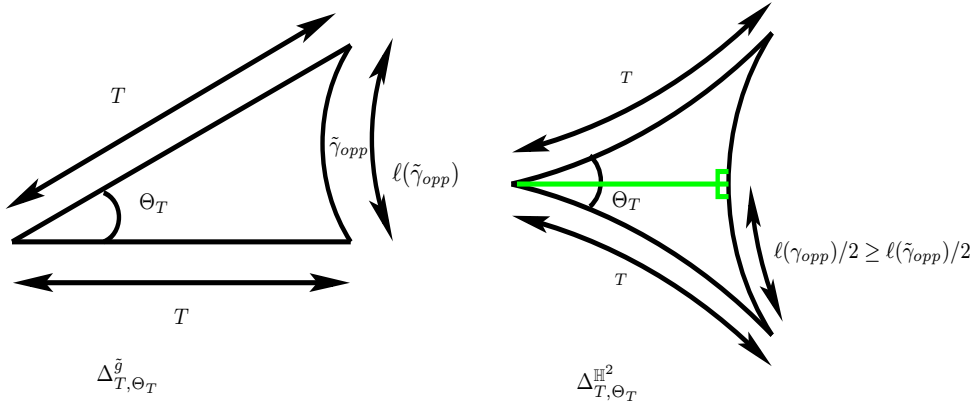
}
\caption{Rauch hinges}
\label{Rauch}
\end{figure}

Consider as well, an isosceles triangle, $\triangle^{{\mathbb H}^2}_{T,\Theta_T}$, in two-dimensional hyperbolic space, ${\mathbb H}^2$, having two sides of equal
length $T$, angle $\Theta_T$ at the associated vertex and  ``Rauch hinge" $\gamma_{opp}$, with length
$\ell(\gamma_{opp})$.  By Toponogov's theorem (see \cite[Theorem 2.2 (B)]{CE}), since we are assuming that the
sectional curvatures of $(\Rn,\tilde g)$ satisfy $-1\le K\le 0$, we must have
$$\ell(\gamma_{opp})\ge \ell(\tilde \gamma_{opp})\ge R.$$
By properties of isosceles triangles in ${\mathbb H}^2$, the ray bisecting the triangle at the vertex spanned by the two
sides of equal length $T$ must intersect the Rauch hinge, $\gamma_{opp}\in \triangle^{{\mathbb H}^2}_{T,\Theta_T}$, orthogonally at its midpoint (see Figure~\ref{Rauch}).  Consequently,
by the law of sines for hyperbolic space, we must have
$$\sin\tfrac12 \Theta_T =\frac{\sinh(\ell(\gamma_{opp})/2)}{\sinh T}\ge \frac{\sinh \tfrac12 R}{\sinh T}= \sin\tfrac12 \theta_T.$$
Thus, \eqref{K6''} is valid and the proof of Proposition~\ref{topprop} is complete. \qed


\newsection{Proof of  Proposition~\ref{tatprop}} \hspace*{\fill} \\

As a first step in the proof of Proposition~\ref{tatprop}, let us make a preliminary reduction.
If $\beta$ is a Littlewood-Paley bump function as in \eqref{littlewoodpaley}, we claim
that it suffices to prove the following dyadic versions of the two inequalities in the Proposition:
\begin{equation}\label{E1}
\Bigl\|\int a(t) e^{-it\la}  \bigl(I-Q_{\theta,\la}\bigr)\circ \beta(P/\la)  e^{itP}f\, dt\Bigr\|_{L^2(\tube)}
\le C_{\delta_0}\la^{-\frac14}\theta^{-\frac12}\|f\|_{L^2(M)},
\end{equation}
and
\begin{equation}\label{E2}
\Bigl\|\int a(t) e^{-it\la}  \bigl(I-Q_{\theta,\la}\bigr) \circ \beta(P/\la) e^{itP}f\, dt\Bigr\|_{L^2(\gamma)}
\le C_{\delta_0,\e}\theta^{-\frac12-\e}\|f\|_{L^2(M)},
\end{equation}
assuming, as in the Proposition that $\theta \ge \la^{-\frac12+\delta_0}$ with $\delta_0$, $\e>0$,
and that $c_0>0$ in \eqref{dyadic} is sufficiently small.

To verify this claim, we note that
\begin{equation}\label{E3}
\int a(t)e^{-i\la t}\, \bigl(1-\beta\bigr)(P/\la)e^{itP} \, dt
\end{equation}
has kernel
$$\sum \widehat a(\la-\la_j)\bigl(1-\beta\bigr)(\la_j/\la) \, e_j(x)\overline{e_j(y)}.$$
Since $a\in C^\infty_0(\R)$ and $\beta$ is as in \eqref{littlewoodpaley},
$$|\widehat a(\la-\la_j) (1-\beta)(\la_j/\la)|\le C(1+\la+\la_j)^{-n-2}.$$
Since, by the Weyl formula, we have
$$\sum_{\la_j\in [\la,\la+1]} |e_j(x)e_j(y)|\le C(1+\la)^{n-1},$$
we conclude that the kernel of the operator given by \eqref{E3} is
$O(\la^{-1})$.  By the first part of \eqref{kernel}, the same is true for the 
kernel of $Q_{\theta,\la}$ composed on the left with this operator, and therefore
$$\int a(t)e^{-i\la t}\, (I-Q_{\theta,\la})\circ \bigl(1-\beta\bigr)(P/\la)e^{itP} \, dt$$
must have a $O(\la^{-1})$ kernel.  This means that it enjoys better bounds than
those posited in Proposition~\ref{tatprop}, which gives us our claim that we just
need to prove \eqref{E1} and \eqref{E2}.

These two inequalities will be a simple consequence
of the following

\begin{lemma}\label{tatlemma}
Fix $0<\delta_0<1/2$, and assume that $\theta\ge \la^{-\frac12+\delta_0}$ and that $\beta$ is as in \eqref{littlewoodpaley}.  Then we have the following estimates which are uniform in $\gamma\in
\varPi$.  First, there is a uniform $C_0<\infty$ so that for large $\la$ we have
\begin{multline}\label{K29}
\Bigl|\Bigl((I-Q_{\theta,\la})\circ  \beta(P/\la)e^{itP}\circ (I-Q_{\theta,\la})^*\Bigr)(x,y)\Bigr|
\le C_{N,\delta_0} \la^{-N} \quad \forall \, N,
\\ \text{if } \, C_0\la^{-\frac12}\theta^{-1}\le |t|\le 1, \, \, \, \text{and } \, \, x,y\in \tube.
\end{multline}
Also, fix $0<\e<\delta_0$.  Then for $s,s'\in [0,1]$ and large $\la$ we have
\begin{multline}\label{K29'}
\Bigl|\Bigl((I-Q_{\theta,\la})\circ \beta(P/\la)e^{itP}\circ (I-Q_{\theta,\la})^*\Bigr)(\gamma(s),\gamma(s'))\Bigr|
\le C_{N,\delta_0}\la^{-N} \quad \forall \, N,
\\ \text{if }  \, \, \la^{-1}\theta^{-2-\e}\le |s-s'|\le 1.
\end{multline}
\end{lemma}

\begin{REMARK}There is a simple explanation for the lower bounds for the
two time scales in this lemma.  The one in \eqref{K29} is
a lower bound for
the escape time for balls of radius $(\theta\la)^{-1}$ to exit $\tube$ if they are traveling 
along geodesics forming an angle $\ge \theta$ from its center, $\gamma$.  Apart from the the $\e>0$,
the lower bound for $|s-s'|$ in \eqref{K29'} is a lower bound for them to escape from $\gamma$.
See Figure~\ref{Escape}.  These escape times are much smaller than $1$ due to our
assumption that $\theta\ge \la^{-\frac12+\delta_0}$.  Also, as $\theta$ increases the ball radius and the
escape route both become more favorable.
\end{REMARK}


\begin{figure}[h]
\resizebox{4.5in}{1.5in}{
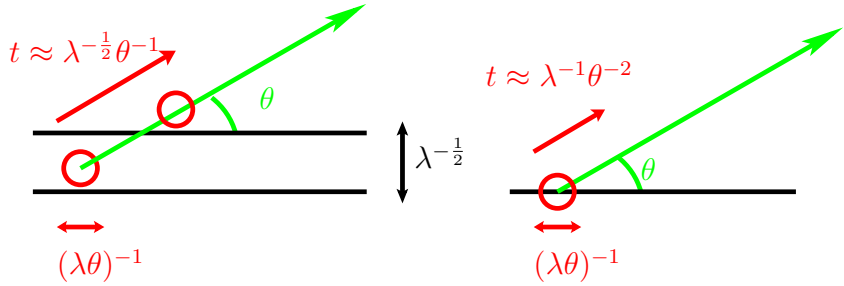
}
\caption{Escape times}\label{Escape}
\end{figure}


\begin{proof}[Proof of \eqref{E1} and \eqref{E2}]
By using smooth cutoffs in the $t$ variable and the fact that the half-wave operators
$e^{itP}$ are bounded on $L^2(M)$ with norm $1$, we deduce that in proving
\eqref{E1} and \eqref{E2}, we may assume that the smooth function $a$ there
is supported in $(-1/10,1/10)$.  The first inequality, \eqref{E1}, then is just the statement that
$$\bigl\|(I-Q_{\theta,\la})\widehat a(P-\la) \beta(P/\la)\bigr\|_{L^2(M)\to L^2(\tube)}
=O(\la^{-\frac14}\theta^{-\frac12}).$$
If $\widetilde a = a(\cd)*\overline{a(-\cd)}$ and $\widetilde \beta=|\beta|^2$, then
a routine $TT^*$ argument shows that this is equivalent
to
\begin{align}\label{K30}
\bigl\|(I-Q_{\theta,\la})&\circ |\widehat a|^2(P-\la) \, \widetilde \beta(P/\la) \circ (I-Q_{\theta,\la})^*f\bigr\|_{L^2(\tube)}
\\
&=\Bigl\|\int \widetilde a(t) e^{-i\la t} (I-Q_{\theta,\la}) \circ \widetilde\beta(P/\la)e^{itP} \circ (I-Q_{\theta,\la})^* f
\, dt\Bigr\|_{L^2(\tube)}
\notag
\\ 
&\le C\la^{-\frac12}\theta^{-1}\|f\|_{L^2}, \quad \text{if } \, \, \, \text{supp }f\subset \tube. \notag
\end{align}
The function $\widetilde a$ is supported in $|t|\le 1$ and $\widetilde \beta$ is as in \eqref{littlewoodpaley} if $\beta$
satisfies those conditions.  Therefore, we can use \eqref{K29} to see
that for $f$ supported in $\tube$ we have that the left side of this inequality is dominated
by
\begin{multline*}
\int_{|t|\le C_0\la^{-\frac12}\theta^{-1}}  \Bigl\|(I-Q_{\theta,\la})\circ \widetilde \beta(P/\la) e^{itP}\circ
(I-Q_{\theta,\la})^*f \Bigr\|_{L^2(M)} \, dt \,  \, + \, \, O(\la^{-N})\|f\|_{L^2(M)}
\\
\lesssim \la^{-\frac12}\theta^{-1}\|f\|_2,
\end{multline*}
using in the last step the fact that $e^{itP}$ is unitary and the fact that the operators
$(I-Q_{\theta,\la})$ and $(I-Q_{\theta,\la})^*$ are uniformly bounded on $L^2$ for $\theta$
as above by virtue of \eqref{kernel}.  Hence, we have \eqref{K30}.

To prove \eqref{E2}, we use another $TT^*$ argument to see that the inequality is 
equivalent to the statement that for $n=2$ we have
\begin{equation}\label{K31}
\Bigl(\int_0^1\Bigl|\int_0^1
K\bigl(\gamma(s),\gamma(s')\bigr) \, h(s')\, ds'\Bigr|^2 ds \Bigr)^{1/2}
\le C\theta^{-1-\e}\|h\|_{L^2([0,1])},
\end{equation}
if 
$$K\bigl(\gamma(s),\gamma(s')\bigr)=\int \widetilde a(t)e^{-i\la t} \Bigl((I-Q_{\theta,\la})\circ \widetilde \beta(P/\la) e^{itP}\circ
(I-Q_{\theta,\la})^*\Bigr)(\gamma(s),\gamma(s')) \, dt.$$
Choose $\rho\in C^\infty_0(\Rn)$ now satisfying $\rho(s)=1$, $|s|\le 1$, and 
$\text{supp }\rho\subset [-2,2]$.  Then, by \eqref{K29'}, modulo an $O(\la^{-N})\|h\|_2$ error,
the left side of \eqref{K31} equals
\begin{equation}\label{K32}\Bigl(\int_0^1\Bigl|\int_0^1
K_{\la,\theta}\bigl(\gamma(s),\gamma(s')\bigr) \, h(s')\, ds'\Bigr|^2 ds\Bigr)^{1/2},\end{equation}
if
\begin{multline}\label{K33}
K_{\la,\theta}(\gamma(s),\gamma(s'))
\\
=\int  \rho(\la \theta^{2+\e}t)  \, \widetilde a(t)e^{-i\la t} \bigl((I-Q_{\theta,\la})\circ \widetilde \beta(P/\la) e^{itP}\circ
(I-Q_{\theta,\la})^*\bigr)(\gamma(s),\gamma(s')) \, dt.\end{multline}
Our assumptions give that $\la \theta^{2+\e}\gg 1$, and so the above integral is just over
a small interval of length $\lesssim \la^{-1}\theta^{-2-\e}$.  We claim that for parameters
as above we have the uniform estimates
\begin{multline}\label{K34}
|K_{\la,\theta}(\gamma(s),\gamma(s'))|\le C\la^{\frac12}  |s-s'|^{-\frac12}, \\ \text{if } \, \, 
s,s'\in [0,1], \, \, \quad \text{and } \, \, |s-s'|\le 10\la^{-1}\theta^{-2-\e},
\end{multline}
and
\begin{multline}\label{K35}
|K_{\la,\theta}(\gamma(s),\gamma(s'))|\le 
 C_N\la^{-N} \, \, \forall \, N, \\ \text{if } \, \, \, 
\, \, \, 
s,s'\in [0,1], \, \, \quad\text{and } \, \, 
|s-s'| \ge 10\la^{-1}\theta^{-2-\e}.
\end{multline}

If we had this we would obtain \eqref{K31}.  For by \eqref{K34}--\eqref{K35} and the 
fact that \eqref{K32} agrees with the left side of \eqref{K31} up to trivial errors,
we conclude that, up to such errors, the left side of \eqref{K31} is dominated by
\begin{equation}\label{K31'}\tag{5.7$'$}
\la^{\frac12} \, \Bigl(\int_0^1\Bigl|\int_{|s-s'|\le 10\la^{-1}\theta^{-2-\e}} h(s') \, |s-s'|^{-\frac12}\, ds'
\Bigr|^2\, ds\Bigr)^{1/2} \le C\theta^{-1-\e}\|h\|_{L^2},
\end{equation}
as desired, by Young's inequality.

We shall postpone the proof of \eqref{K34}--\eqref{K35} for the moment since it follows
from arguments that we shall use in the proof of Lemma~\ref{tatlemma}.
\end{proof}

\begin{proof}[Proof of Lemma~\ref{tatlemma}]
We need to prove \eqref{K29} and \eqref{K29'}.  The proof is very similar to the proof
of (3-14) in \cite{BS}.

To prove \eqref{K29}, let us first
assume that the $x\in \tube$ lies on our $\gamma\in \varPi$.  The kernel in the left
side of \eqref{K29} is a Lagrangian distribution in $y$.  If we choose geodesic normal
coordinates vanishing at $x$ so that $\gamma$ is part of the first coordinate axis
then, assuming as we are in \eqref{K29} that $y$ is of distance $\le 1$ from
$x$, modulo $O(\la^{-N})$ errors, 
if the $c_0$ in \eqref{dyadic} is sufficiently small,
this kernel is of the form
\begin{equation}\label{K36}I_t(y)=\int_{\Rn} e^{-iy\cdot \xi+it|\xi|} \, a_{\theta,\la}(t,y,\xi) \, d\xi,
\end{equation}
with amplitude satisfying
\begin{equation}\label{K9'} 
a_{\theta,\la}(t,y,\xi)=0, \quad \text{if } \, \, \,
|\xi|\notin [\la/10,10\la], \, \, \, \text{or }
 |\xi'|/|\xi|\le 2B\theta, \, \, \text{if }\, \, \xi'=(\xi_2,\dots,\xi_n),
\end{equation}
for some $B$ independent of $\la$, $\theta$,
by virtue of \eqref{PDO} and \eqref{K10}, as well as
\begin{equation}\label{K10'} 
|\partial_t^jD^\alpha_\xi a_{\theta, \la}(t,y,\xi)|\le C_{\alpha,j} (\la/\theta)^{-|\alpha|}.
\end{equation}
We  are using here the fact that if $c_0$ in \eqref{dyadic} is small enough
then the symbol of the dyadic operators $(I-Q_{\theta,\la})\circ \beta(P/\la)$ must vanish in
a cone of aperture $\approx \theta$ about the $\xi_1$-axis because of \eqref{PDO}
and \eqref{dyadic} and the fact that the symbol of $\beta(P/\la)$ is supported
in the region were $|\xi|\approx \la$ by \eqref{littlewoodpaley}.

We would get \eqref{K29} for $x\in \gamma$ if we could show that
\begin{equation}\label{K37}
I_t(y)=O(\la^{-N}), \quad \text{if } \, \, \, |y'| \le C_1 \la^{-\frac12}, \, \, \,
\text{and } \, \, \, |t|\ge C_0\la^{-\frac12}\theta^{-1},
\end{equation}
for some large constant $C_0$ (independent of $\theta$ and $\la$)
since, in the coordinates we are using, $|y'| \le C_1 \la^{-\frac12}$,
is is satisfied with $C_1$ fixed by all points in $\tube$ with coordinates $y$.

To prove this we note that we can obtain favorable lower bounds for the
phase functions for the Lagrangian distributions in \eqref{K36}.
Specifically, we have for points in $\tube$ satisfying the first condition
in \eqref{K37}
\begin{multline*}
\bigl|\nabla_{\xi'}\bigl(-y\cdot \xi+t|\xi|\bigr)\bigr|=\bigl| t\tfrac{\xi'}{|\xi|}-y'\bigr|
\ge  \bigl|\, t\tfrac{\xi'}{|\xi|}\, \bigr| -C_1\la^{-\frac12}
\\
\ge \la^{-\frac12}, \quad \text{if }\, \, \, 
|\xi'|/|\xi|\ge B|\theta|, \, \, \, \text{and } \, \, \, |t|\ge C_0\la^{-\frac12}\theta^{-1},
\end{multline*}
provided that $C_0B\ge 1$.
Using this, by virtue of \eqref{K9'}--\eqref{K10'}, one obtains \eqref{K37} by a simple integration by parts argument.  Indeed, by \eqref{K10'}, 
each integration by parts yields a gain of $\approx \la^{-\frac12}/\theta
\ge \la^{-\delta_0}$, since we are assuming that $\theta\ge \la^{-\frac12+\delta_0}$.

Thus, we have proved \eqref{K29} when $x$ lies on $\gamma$.  The same argument will
work for $x\in \tube$ since $\la^{-\frac12} \ll \theta$.  Given any $x\in \tube$,  pick the geodesic segment
$\gamma_x\in \varPi$ of closest distance to $\gamma$, using the standard metric on $S^*M$ and identifying
$\varPi$ with $S^*M$.
Because of the relationship between $\theta$ and $\la^{-\frac12}$ the operators
$Q_{\theta,\la}$ also localize to a  $\approx \theta$ neighborhood of $S^*\gamma_x$, just as it does for
$S^*\gamma$, due to the fact that
the distance between the central geodesic $\gamma$ and any such $\gamma_x$ must be
$O(\la^{-\frac12})$ when $x\in \tube$.  So if we choose geodesic normal coordinates
as above about $x$, we still can write the kernel as a function of the other variables
as in \eqref{K36}.  So the argument for the special case that we have just completed
implies that the same bounds are valid for all $x\in \tube$, which completes the proof
of \eqref{K29}.

The proof of \eqref{K29'} is very similar.  If we use geodesic normal coordinates
about  $\gamma(s)$ as above then, modulo trivial errors, the kernels in the left side of \eqref{K29'}
are of the form
\begin{equation}\label{K36'}\tag{5.5$'$}
I_t(s,s') =\int_{{\mathbb R}^2} e^{-i(s-s')\xi_1 +it|\xi|} a_{\theta,\la}(t,s,s',\xi)\, d\xi,
\end{equation}
with amplitudes satisfying \eqref{K9'}--\eqref{K10'} (with $y$ replaced by $(s,s')$).
Clearly if $|\xi_2 |/|\xi|\ge B\theta$ and if $|t|\ge \tfrac12\la^{-1}\theta^{-2-\e}$, we have
$$\bigl|\tfrac\partial{\partial \xi_2} \bigl(-(s-s')\xi_1+t |\xi| \bigr) \bigr|\ge B\la^{-1}\theta^{-1-\e}.$$
Thus, by \eqref{K9'}, every time we integrate by parts in $\xi_2$ we gain
by $\theta^{\e} =\la^{-\delta_0\e}$, and, therefore, since we are assuming that $\delta_0$ and $\e$
are both positive, we have for large $\la$
$$I_t(s,s')=O(\la^{-N}), \quad \text{if } \, \, \, |t|\ge \tfrac12\la^{-1}\theta^{-2-\e},$$
which means that we have \eqref{K29'} with no condition on $s,s'\in [0,1]$,
if $|t|\ge \tfrac12 \la^{-1}\theta^{-2-\e}$.  If 
we integrate by parts in $\xi_1$ we find that we also have
\begin{equation}\label{K38}I_t(s,s')=O(\la^{-N}), \quad \text{if } \, \, \, |t|\le \tfrac12 \la^{-1}\theta^{-2-\e},
\, \, \, \text{and } \, \, |s-s'|\ge \la^{-1}\theta^{-2-\e},\end{equation}
which yields the remaining part of \eqref{K29'}.
\end{proof}

\begin{proof}[Proof of \eqref{K34}--\eqref{K35}]
The kernels involved are just
\begin{multline}\label{S17}
\int \rho(\la \theta^{2+\e}t) \widetilde a(t) e^{-i\la t}I_t(s,s')\, dt
\\
=\iint \rho(\la\theta^{2+\e}t) \tilde a(t) e^{i(s-s')\xi_1} a_{\theta,\la}(t,s,s',\xi) \, e^{it(|\xi|-\la)} \, dt d\xi.
\end{multline}

In view of the support properties of $\rho$, \eqref{K35} follows from the proof
of \eqref{K38}, since the latter shows that $I_t(s,s')$ is rapidly decreasing
if $|s-s'|\ge 2|t|$ and if $|t|\lesssim \la^{-1}\theta^{-2-\e}$.

To prove \eqref{K34}, 
let us first note that if $\widetilde K_{\theta,\la}$ denotes the analog of the kernel in \eqref{K33} but without
the $(I-Q_{\theta,\la})$ operators in the left and right, i.e.,
$$\widetilde K_{\theta,\la}\bigl(\gamma(s),\gamma(s')\bigr)=\int \rho(\la\theta^{2+\e}t) \widetilde a(t) e^{-i\la t}\bigl(\widetilde \beta(P/\la) 
e^{itP}\bigr)(\gamma(s),\gamma(s')) \, dt,$$
then
\begin{equation}\label{K34'}\tag{5.10$'$}
|\widetilde K_{\theta,\la}(\gamma(s),\gamma(s'))| =O\bigl(\la^{\frac12}|s-s'|^{-\frac12}\bigr).
\end{equation}
To see this we use the Hadamard parametrix and the calculus of Fourier integrals to see that modulo a $O(\la^{-N})$ error term
$$\bigl(\widetilde \beta(P/\la) 
e^{itP}\bigr)(\gamma(s),\gamma(s'))=\int_{{\mathbb R}^2}e^{i(s-s')\xi_1+it|\xi|}\alpha(t,s,s',|\xi|) \, d\xi,$$
where $\alpha$ is a zero-order symbol.  Thus, modulo trivial errors,
\begin{multline}\label{S18}
\widetilde K_{\theta,\la}(\gamma(s),\gamma(s'))
=  
\\
\int \int_0^\infty e^{it(r-\la)} \rho(\la\theta^{2+\e}t) \alpha(t,s,s',r)r \, \Bigl( \int_{S^1} e^{ir(s-s') \langle (1,0),\omega \rangle} \,
 d\omega \, \Bigr) \, drdt.
\end{multline}
Integrating by parts in $t$ shows that this expression  is majorized by
\begin{multline*}
\int_0^\infty (\la \theta^{2+\e})^{-1} \, \bigl(1+\la^{-1}\theta^{-2-\e}|r-\la| \bigr)^{-3} \, rdr
\\
=\la \int_0^\infty \theta^{-2-\e}\, \bigl(1+\theta^{-2-\e}|r-1|\bigr)^{-3} \, rdr =O(\la), 
\end{multline*}
since $\theta^{-2-\e}>1$.  Thus, \eqref{K34'} is valid when $|s-s'|\le \la^{-1}$.  To handle the remaining case, we 
recall that, by stationary phase,
\begin{equation}\label{5.19}
\int_{S^1}e^{ix\cdot \omega}\, d\omega =O(|x|^{-\frac12}), \quad |x|\ge1.
\end{equation}
If we plug this 
into \eqref{S18} with $x=r(s-s',0)$, and, as above, integrate by parts in $t$, we conclude that when $\la^{-1}\lesssim |s-s'|$ we have
\begin{multline*}
|\widetilde K_{\theta,\la}(\gamma(s),\gamma(s'))| \lesssim \int_0^\infty (\la\theta^{2+\e})^{-1} \, \bigl(1+\la^{-1}\theta^{-2-\e}|r-\la|\bigr)^{-3}
\, \bigl(r|s-s'|\bigr)^{-\frac12} \, rdr
\\
=O\bigl(\la^{\frac12}|s-s'|^{-\frac12}\bigr),
\end{multline*}
as claimed, which finishes the proof of \eqref{K34'}.

Next, since the symbol of $Q_{\theta,\la}$ is as in \eqref{PDO}, it follows that if the constant $c_0>0$ in \eqref{dyadic} is small enough
then $Q_{\theta,\la}\circ \widetilde\beta(P/\la)$ and $\widetilde\beta(P/\la)\circ Q_{\theta,\la}^*$ both have symbols supported in the set
where $|\xi|\approx \la$ and $|\xi_1|/|\xi|\le C\theta$ for some constant $C$.  Thus, it follows that $K_{\theta,\la}-\widetilde K_{\theta,\la}$
can be written in the form \eqref{S17} except with $I_t(s,s')$ replaced by
$$\widetilde I_t(s,s')=\int_{{\mathbb R}^2} e^{i(s-s')\xi_1+it|\xi|} \widetilde a_{\theta,\la}(t,s,s',\xi) \, d\xi,$$
where the amplitude here satisfies the bounds in \eqref{K10'}, but in place of \eqref{K9'}, we have
\begin{equation}\label{S20}
\widetilde a_{\theta,\la}(t,s,s',\xi)=0\, \, \, \text{if } \, \, |\xi|\notin [\la/10,10\la] \, \, \, \text{or } \, \, |\xi_2|/|\xi|\ge C\theta.
\end{equation}
  From this,   and straighforward analogs of the above calculations, we deduce that, modulo trivial errors,
\begin{multline*}
\bigl|K_{\theta,\la}((\gamma(s),\gamma(s'))-\widetilde K_{\theta,\la}((\gamma(s),\gamma(s'))\bigr|
\\
\lessim (\la\theta^{2+\e})^{-1}
\int_{ \{\xi\in {\mathbb R}^2: \, |\xi_2|\le C\theta|\xi|\}     } 
\bigl(1+\la^{-1}\theta^{-2-\e}|\, |\xi|-\la \, |\bigr)^{-3} d\xi
 =O(\la\theta).
\end{multline*}
As $\la\theta\le \la^{\frac12}|s-s'|^{-\frac12}$ if $|s-s'|\le \la^{-1} \theta^{-2}$, by \eqref{K34'}, it suffices to show that
\begin{multline*}
\int \int_{{\mathbb R}^2}\rho(\la \theta^{2+\e}t) \, e^{it(|\xi|-\la)} e^{i(s-s')\xi_1} \, \widetilde a_{\theta,\la}(t,s,s',\xi)  \, d\xi dt
\\
=O(\la^{\frac12}|s-s'|^{-\frac12}), \quad \text{if } \, \, |s-s'|\ge \la^{-1}\theta^{-2}.
\end{multline*}
But this follows from the argument that we just used for $\widetilde K_{\theta,\la}$ if, instead of \eqref{5.19}, we use the fact that for each
fixed $j=0,1,\dots$,
\begin{equation}\label{5.19'}\tag{5.19$'$}
\int_{S^1} e^{irx\cdot \omega} \,  \partial_t^j \widetilde a_{\theta,\la}(t,s,s',r\omega) \, d\omega
=O\bigl((\la|x|)^{-\frac12}\bigr), \quad \text{if } \, \, r\approx \la, \, \, \, \text{and } \, \, \la^{-1}\theta^{-2}\lesssim |x|,
\end{equation}
due to \eqref{K9'} and \eqref{K10'}.
Since the $\widetilde a_{\theta,\la}$ satisfy the bounds in \eqref{K10'}, this follows from the method of stationary phase
by slightly modifying the standard proof of \eqref{5.19}.
\end{proof}


\section*{Acknowledgements}
With pleasure, we thank William Minicozzi and Steve Zelditch for many helpful discussions and
for generously sharing their knowledge.

\end{document}

%% file: rauch_hinges.pdf_t
\begin{picture}(0,0)%
\includegraphics{rauch_hinges.pdf}%
\end{picture}%
\setlength{\unitlength}{3947sp}%
\begingroup\makeatletter\ifx\SetFigFont\undefined%
\gdef\SetFigFont#1#2#3#4#5{%
  \reset@font\fontsize{#1}{#2pt}%
  \fontfamily{#3}\fontseries{#4}\fontshape{#5}%
  \selectfont}%
\fi\endgroup%
\begin{picture}(5378,2365)(1216,-2243)
\put(3901,-961){\makebox(0,0)[lb]{\smash{{\SetFigFont{9}{10.8}{\rmdefault}{\mddefault}{\updefault}{\color[rgb]{0,0,0}$\ell(\tilde{\gamma}_{opp})$}%
}}}}
\put(5457,-384){\makebox(0,0)[lb]{\smash{{\SetFigFont{6}{7.2}{\rmdefault}{\mddefault}{\updefault}{\color[rgb]{0,0,0}$T$}%
}}}}
\put(5420,-1572){\makebox(0,0)[lb]{\smash{{\SetFigFont{6}{7.2}{\rmdefault}{\mddefault}{\updefault}{\color[rgb]{0,0,0}$T$}%
}}}}
\put(5160,-2188){\makebox(0,0)[lb]{\smash{{\SetFigFont{8}{9.6}{\rmdefault}{\mddefault}{\updefault}{\color[rgb]{0,0,0}$\Delta^{\mathbb{H}^2}_{T,\Theta_T}$}%
}}}}
\put(6511,-1415){\makebox(0,0)[lb]{\smash{{\SetFigFont{8}{9.6}{\rmdefault}{\mddefault}{\updefault}{\color[rgb]{0,0,0}$\ell(\gamma_{opp})/2\geq\ell(\tilde{\gamma}_{opp})/2$}%
}}}}
\put(3470,-788){\makebox(0,0)[lb]{\smash{{\SetFigFont{8}{9.6}{\rmdefault}{\mddefault}{\updefault}{\color[rgb]{0,0,0}$\tilde{\gamma}_{opp}$}%
}}}}
\put(2167,-1126){\makebox(0,0)[lb]{\smash{{\SetFigFont{8}{9.6}{\familydefault}{\mddefault}{\updefault}{\color[rgb]{0,0,0}$\Theta_T$}%
}}}}
\put(1973,-498){\makebox(0,0)[lb]{\smash{{\SetFigFont{8}{9.6}{\rmdefault}{\mddefault}{\updefault}{\color[rgb]{0,0,0}$T$}%
}}}}
\put(2408,-1802){\makebox(0,0)[lb]{\smash{{\SetFigFont{8}{9.6}{\rmdefault}{\mddefault}{\updefault}{\color[rgb]{0,0,0}$T$}%
}}}}
\put(1877,-2140){\makebox(0,0)[lb]{\smash{{\SetFigFont{8}{9.6}{\rmdefault}{\mddefault}{\updefault}{\color[rgb]{0,0,0}$\Delta^{\tilde{g}}_{T,\Theta_T}$}%
}}}}
\put(5401,-1111){\makebox(0,0)[lb]{\smash{{\SetFigFont{8}{9.6}{\familydefault}{\mddefault}{\updefault}{\color[rgb]{0,0,0}$\Theta_T$}%
}}}}
\end{picture}%

%% file: escape_times.pdf_t
\begin{picture}(0,0)%
\includegraphics{escape_times.pdf}%
\end{picture}%
\setlength{\unitlength}{3947sp}%
\begingroup\makeatletter\ifx\SetFigFont\undefined%
\gdef\SetFigFont#1#2#3#4#5{%
  \reset@font\fontsize{#1}{#2pt}%
  \fontfamily{#3}\fontseries{#4}\fontshape{#5}%
  \selectfont}%
\fi\endgroup%
\begin{picture}(5298,1827)(136,-1255)
\put(3451,-1186){\makebox(0,0)[lb]{\smash{{\SetFigFont{12}{14.4}{\rmdefault}{\mddefault}{\updefault}{\color[rgb]{1,0,0}$(\lambda\theta)^{-1}$}%
}}}}
\put(451,-1186){\makebox(0,0)[lb]{\smash{{\SetFigFont{12}{14.4}{\rmdefault}{\mddefault}{\updefault}{\color[rgb]{1,0,0}$(\lambda\theta)^{-1}$}%
}}}}
\put(1726,-136){\makebox(0,0)[lb]{\smash{{\SetFigFont{12}{14.4}{\rmdefault}{\mddefault}{\updefault}{\color[rgb]{0,1,0}$\theta$}%
}}}}
\put(4126,-586){\makebox(0,0)[lb]{\smash{{\SetFigFont{12}{14.4}{\rmdefault}{\mddefault}{\updefault}{\color[rgb]{0,1,0}$\theta$}%
}}}}
\put(2701,-511){\makebox(0,0)[lb]{\smash{{\SetFigFont{12}{14.4}{\rmdefault}{\mddefault}{\updefault}{\color[rgb]{0,0,0}$\lambda^{-\frac 12}$}%
}}}}
\put(151,164){\makebox(0,0)[lb]{\smash{{\SetFigFont{12}{14.4}{\rmdefault}{\mddefault}{\updefault}{\color[rgb]{1,0,0}$t\approx\lambda^{-\frac 12}\theta^{-1}$}%
}}}}
\put(3151, 14){\makebox(0,0)[lb]{\smash{{\SetFigFont{12}{14.4}{\rmdefault}{\mddefault}{\updefault}{\color[rgb]{1,0,0}$t\approx \lambda^{-1}\theta^{-2}$}%
}}}}
\end{picture}%

%% file: TopBS.bbl
\begin{thebibliography}{MA}
\bibitem{Berard}  B\'erard, P., {\em On the wave equation on a compact Riemannian manifold without conjugate
 points},
 Math. Z.  {\bf 155}  (1977),  249--276.
 


\bibitem{BS}  Blair, M. D. and  Sogge, C. D.,  {\em Refined and microlocal Kakeya-Nikodym bounds for eigenfunctions in two
 dimensions},
 Anal. PDE  {\bf 8}  (2015),   747--764.
 
 \bibitem{BS2}  Blair, M. D. and  Sogge, C. D.,
 {\em On Kakeya-Nikodym averages, $L^p$-norms and lower bounds for nodal sets of eigenfunctions in higher dimensions},
 J. European Math. Soc., {\bf 17} (2015), 2513--2543.
 
  \bibitem{BS15}  Blair, M. D. and Sogge, C. D., {\em Refined and microlocal Kakeya-Nikodym bounds of eigenfunctions
 in higher dimensions}, preprint.

\bibitem{Bo} Bourgain, J., {\em  Geodesic restrictions and $L^p$-estimates for eigenfunctions of
 Riemannian surfaces},
 ``Linear and complex analysis,"
 27--35, Amer. Math. Soc. Transl. Ser. 2, {\bf 226}, Amer. Math. Soc., Providence, RI,  2009.

\bibitem{BGT} Burq, N.,   G\'erard, P. and Tzvetkov, N.,  {\em Restrictions of the Laplace-Beltrami eigenfunctions to
 submanifolds},
 Duke Math. J.  {\bf 138}  (2007), 445--486.
 
 \bibitem{Chavel} Chavel, I., {\em Riemannian geometry, 
A modern introduction, 
Second edition,} 
Cambridge Studies in Advanced Mathematics, {\bf 98}. Cambridge University Press, Cambridge,  2006.
 
 \bibitem{CE}  Cheeger, J. and  Ebin, D. G.,  {\em Comparison theorems in Riemannian geometry, 
Revised reprint of the 1975 original},
AMS Chelsea Publishing, Providence, RI,  2008.
 
\bibitem{Ch} Chen, X., {\em An improvement on eigenfunction restriction estimates for compact
 boundaryless Riemannian manifolds with nonpositive sectional curvature},
 Trans. Amer. Math. Soc.  {\bf 367}  (2015),   4019--4039.

\bibitem{CS} Chen, X. and  Sogge, C. D.,  {\em  A few endpoint geodesic restriction estimates for eigenfunctions},
 Comm. Math. Phys.  {\bf 329}  (2014),   435--459.
 
 \bibitem{CM} Colding, T. H. and  Minicozzi, W. P., II, {\em Lower bounds for nodal sets of eigenfunctions},
 Comm. Math. Phys.  {\bf 306}  (2011),  777--784.
 
 \bibitem{DF} Donnelly, H.; Fefferman, C. {\em Nodal sets of eigenfunctions on Riemannian manifolds},  Invent. Math. {\bf 93} (1988), 161--183.

 
 \bibitem{GS} Greenleaf, A. and  Seeger, A., {\em Fourier integral operators with fold singularities},
 J. Reine Angew. Math.  {\bf 455}  (1994), 35--56.
 
  \bibitem{H}  Han, X., {\em  Small scale quantum ergodicity on negatively curved manifolds}, arXiv:1410.3911 (2014).
  
  \bibitem{HT} Hassell, A.  and Tacy, M., {\em Improvement of eigenfunction estimates on manifolds of nonpositive
 curvature},
 Forum Math.  {\bf 27}  (2015),   1435--1451.

\bibitem{HR}  Hezari, H. and  Rivi\`ere, G., {\em $L^p$ norms, nodal sets, and quantum ergodicity}, arXiv:1411.4078 (2014).


\bibitem{HS} Hezari, H. and  Sogge, C. D., {\em A natural lower bound for the size of nodal sets},
 Anal. PDE  {\bf 5}  (2012),   1133--1137.
 
 \bibitem{Ral} Ralston, J. {\em Approximate eigenfunctions of the Laplacian},
 J. Differential Geometry  {\bf 12}  (1977),  87--100.

 
 \bibitem{Seig}  Sogge, C. D., {\em Concerning the $L^p$ norm of spectral clusters for second-order
 elliptic operators on compact manifolds},
 J. Funct. Anal.  {\bf 77}  (1988), 123--138.
 
 \bibitem{SFIO} Sogge, C. D.,  {\em Fourier integrals in classical analysis},
Cambridge Tracts in Mathematics, {\bf 105} Cambridge University Press, Cambridge,  1993.

\bibitem{SKN} Sogge, C. D.,  {\em Kakeya-Nikodym averages and $L^p$-norms of eigenfunctions},
 Tohoku Math. J. {\bf 63}  (2011),  519--538.
 
 
 \bibitem{Sball} Sogge, C. D.,  {\em Localized $L^p$-estimates of eigenfunctions: A note on an article of Hezari and Rivi\`ere},
arXiv:1503.07238.

\bibitem{STZ} Sogge, C. D.,  Toth, J. A. and  Zelditch, S., {\em About the blowup of quasi-modes on Riemannian manifolds},
 J. Geom. Anal.  {\bf 21}  (2011),  150--173.

\bibitem{SZ} Sogge, C. D. and Zelditch, S., {\em Riemannian manifolds with maximal eigenfunction growth}, 
 Duke Math. J.  {\bf 114}  (2002),  387--437.
 
  \bibitem{SZ11} Sogge, C. D. and Zelditch, S., {\em Lower bounds on the Hausdorff measure of nodal sets}
 Math. Res. Lett.  {\bf 18}  (2011),  25--37.
 
\bibitem{SZnod2}  {\em Lower bounds on the Hausdorff measure of nodal sets II},
 Math. Res. Lett.  {\bf 19}  (2012),  1361--1364.
 
\bibitem{SZKN} Sogge, C. D. and  Zelditch, S., { \em On eigenfunction restriction estimates and $L^4$-bounds for compact
 surfaces with nonpositive curvature}, 
``Advances in analysis: the legacy of Elias M. Stein'', 
 447--461, Princeton Math. Ser., {\bf 50}, Princeton Univ. Press, Princeton, NJ,  2014.
 \bibitem{SZRA}  Sogge, C. D. and Zelditch, S.,
{\em Focal points and sup-norms of eigenfunctions},   Rev. Mat. Iberoamericana, to appear.



\bibitem{SZRA2}  Sogge, C. D.  and Zelditch, S., {\em Focal points and sup-norms of eigenfunctions 
manifolds II: the two-dimensional case}, Rev. Mat. Iberoamericana, to appear.

 \bibitem{SZqm}  Sogge, C. D. and Zelditch, S., {\em A note on the $L^p$ norms of quasi-modes}, arXiv:1401.0345.
 
 
 \bibitem{Tat} Tataru, Daniel, {\em On the regularity of boundary traces for the wave equation},
 Ann. Scuola Norm. Sup. Pisa Cl. Sci. {\bf 26}  (1998),  185--206.

 \bibitem{Top}  Toponogov, V. A., {\em Riemann spaces with curvature bounded below}, (Russian), Uspehi Mat. Nauk {\bf 14}
 (1959), 87--130.
 
 \bibitem{Yau} Yau, S.T. ,{\em Survey on partial differential equations in differential geometry. Seminar on Differential Geometry}, pp. 3--71, Ann. of Math. Stud., {\bf 102}, Princeton Univ. Press, Princeton, N.J., 1982.

%
\end{thebibliography}
